\documentclass{article}[11pt]
\usepackage{amsmath,amssymb,amsthm}
\usepackage[dvips]{graphicx}
\usepackage{pgf}

\title{\textbf{{ Invariant Gorenstein Rings On $3$-Dimensional Vector Spaces}}}
\author{
\\
\textbf{{ By: Tamir Buqaie}}\\
\\
\\
\textbf{{ Supervised by: Prof. Amiram Braun}}\\
\\
\\
{\sc Thesis submitted as partial fulfillment of the}\\
{\sc requirements for the masters degree}\\
\\
\\
{\sc University of Haifa}\\
{\sc Faculty of Natural Sciences }\\
{\sc Department of Mathematics}\\
\\
\\
{ October, 2018}\\
\\
\\
\\
{ Approved by: \underline{~~~~~~~~~~~~~~~~~~~~~~~~~~~~~~~~~~~~~~~~~~~~~~~~~} Date: \underline{~~~~~~~~~~~~~~~~~}}\\
(\small Supervisor)\\
{ Approved by: \underline{~~~~~~~~~~~~~~~~~~~~~~~~~~~~~~~~~~~~~~~~~~~~~~~~~} Date: \underline{~~~~~~~~~~~~~~~~~}}\\
(\small Chairperson of M.sc. Committee)\\
}

\date{}
\clearpage
\usepackage{graphicx}
\usepackage{amsthm}
\usepackage{amsmath}

\newtheorem{example}{Example}[section]
\usepackage{tasks}

\newtheorem{theorem}{Theorem}[section]
\usepackage{enumerate}
\usepackage{enumitem}
\newtheorem{remark}[theorem]{Remark}
\newtheorem{lemma}[theorem]{Lemma}
\newtheorem{corollary}[theorem]{Corollary}
\newtheorem{proposition}[theorem]{Proposition}

\newtheorem{assumption}[theorem]{Assumption}
\usepackage{systeme,mathtools}
\usepackage{amssymb}
\usepackage{anyfontsize}
\usepackage[left=30mm]{geometry}
\usepackage{amsmath,amssymb,amsthm}
\usepackage{pgf}

\begin{document}

\setcounter{page}{1}
\pagenumbering{Roman}
\maketitle
\newpage
\vspace{60pt}
\noindent
{\huge Acknowledgments}
\vspace{60pt}

I would like to thank my Master thesis advisor, Professor Amiram Braun  for his dedication, interest, professionality and human attitude.I appreciate his constructive comments which helped me to learn and improve this paper.\\
I would like to thank my first mathematics teacher,Osima Shehady for her creative,successful and outstanding education methods.
I  thank my family and my friends,for their support, understanding and for giving me positive energy.\\\\\\\\

\noindent

\vspace{60pt}
\newpage
{\large
\tableofcontents 

}
\newpage

\addcontentsline{toc}{section}{Abstract}
\begin{center}
     \hskip 1pt \\
     \hskip 1pt \\    
     {\Large \textbf{\underline{Invariant Gorenstein Rings On $3$-Dimensional Vector Spaces}}}\\
     \hskip 1pt \\
     \hskip 1pt \\
     \hskip 1pt \\
{\large  \textbf{   Tamir Buqaie} \\
    \hskip 1pt \\

\textbf{{ Supervised by: Prof. Amiram Braun}}\\}
     \hskip 1pt \\
     \hskip 1pt \\
\end{center}
\begin{center}
{\large \textbf{\underline{Abstract}}} \\
\hskip 1pt \\
\end{center} 
{\large
For a finite subgroup $G\subset SL(V),$ let $S(V)^G$ denote the subring of $G-$invariants, where $V$ is a $3-$ dimensional $G-module$ over $F$ and $S(V)$ the symmetric algebra of $V.$ In this thesis we discuss the modular case of invariant theory,that is,when $char F=p$ divides $|G|.$ The aim of this work is to classify \textbf {Gorenstein} ring of the form $S(V)^G .\\\\$
 A.Braun in \cite{Braun2} and Fleischmann-Woodcock in \cite{FW} proved that if $S(V)^G$ is Cohen-Macaulay ring and $S(V)^{W(G)}$ is a polynomial ring then  $S(V)^G$ is Gorenstein if and only if $G/W(G)\subseteq SL(m/m^2),$ where $W(G)$ is the $G-$subgroup generated by all pseudo-reflections and $m$ is the maximal homogenous ideal of $S(V)^{W(G)}.$ ( See [Theorem \ref{T2}]). \\\\
Here, since  $G\subset SL(V),$  we have one kind of pseudo-reflection ,called transvection, in other words $W(G)=T(G),$ where $T(G)$ denote the $G-$ subgroup  generated by all transvections of $G.$This work is divide into $7$ chapters, the classification of $S(V)^G,$  in each chapter, depends on the previous theorem which provides a necessary and suffiecient condition for $G$ such that $S(V)^G$ is Gorenstein.So to apply this theorem we might ask ourselves,  whether $S(V)^{T(G)}$ is polynomial ring and what are the $T(G)-$invariants of this ring?Or at least, what are the degrees of the $T(G)-$invariants ? Since $dim_FV=3 ,$ the Cohen-Macaulay property of $S(V)^{G}$ holds  
\cite [Proposition 5.6.10]{NS}.These are some of the classical questions that invariant theorists have studied in the modular case.\\\textbf{\underline{We prove the  following theorem }}: Let $G\subset SL(3,\mathbb{F}_p)$ be a finite group and $V=\mathbb{F}_p^3$ is a reducible $G$-module,then $S(V)^G$ is Gorenstein.(See [Theorem \ref{T3}]). 
\newpage
\textbf{\underline{We explain now how this work is organized. }}\\\\
As  mentioned above, we divide our work into $7$ chapters.In chapters $A,B,D,E,F$ and $G$ we fix a basis for $V,$ and represent each $g\in G$,by matrices,with respect to this basis.Then we compute the ring of  $T(G)-$invariants. We are mainly concered with the case of $S(V)^{T(G)}$ being a polynomial ring.Under the assumptions of  each  chapter we prove that $S(V)^{T(G)}$ is a polynomial ring .Generally, we can not confirm in chapter $D$ that the polynomial property of $S(V)^{T(G)}$ holds.We provide in chapters $A,B,D,$ and $G,$ a sufficient and necessary condition on $G$ such that $S(V)^G$ is Gorenstein ,when $G\subset SL(3,F)$.However, in chapter $E$ and $F,$ we  exclusively consider  the case $G\subset SL(3,\mathbb{F}_p).$ The technique  in this work will be to  exhibit the action of the group $G$ on $m/m^2$ and using [Theorem \ref{T2}]. \\\\
In order to prove that $S(V)^{T(G)}$ is a polynomial ring ,in most cases, we use Kemper's theorem \cite[Theorem 3.7.5]{Kemper1}.\\\\
In  chapter $A,$ we consider $V,$ a reducible $G-module$ of the form  $V=Fv
\oplus W$  and $Fv$ , $W$  are $G-$submodules with $dim_{F}W=2.$ This is a special chapter,
since by \textbf{Proposition \ref{AP1}}, the Gorenstein property of $S(V)^{G}$ depends on $W.$ So  we use the computation of \cite{Braun3},and proceed our proof in steps by dividing it into $3$ cases, which are:

\begin{enumerate}
\item $W$ is an irreducible and primitive $T(G)-module.$

\item $W$ is a reducible $T(G)-module.$
\item $W$ is an irreducible and imprimitive $T(G)-module.$

\end{enumerate}
In fact,the results in chapter $A$ are  obtained from the discussion in \cite{Braun3}.The difference between the situation  here  and the situation of
\cite{Braun3}, is that in \cite{Braun3}, $G\subset SL(2,F)$ and here $\bar{G}=G|W\subset GL(2,F).\\$ \\
In  chapter $B,$ we consider $V,$ a reducible,
indecomposable $G-module$ which has at least two $G$-submodules
$W_1,W_2$ where $dimW_i=2$ for $i=1,2.$ In addition to the main theorem, we prove,under specific condition, that   $S(V)^{G}$ is Gorenstein in the following cases (see [Theorem \ref{CT1}]) :

\begin{enumerate}[label=(\roman*)]
\item $p=3.$

\item  $|G|$ is prime to $3.$

 \item $F=\mathbb{F}_q,$ where $q=p^s$ and g.c.d$(3,s)=1.$

\end{enumerate}
 We define a $F-$vector space $X,$ where $dim_FX\in \{1,2\},$ and separate this chapter into two cases.  In the first case, we deal with $dim (span_F X)=1.$ In the second case, we deal with $dim (span_F X)=2,$ where $T(G)$ is a direct product of two transvection groups $T_1(G),T_2(G).$ In \textbf{Lemma \ref{CL3}},\textbf{Lemma \ref{CL4}}, and \textbf{Proposition \ref{CP2}} we examine the normality of $T_1(G),T_2(G).$ However this is  not needed for \textbf{Proposition \ref{CP3}} ,\textbf{Proposition  \ref{CL8}}, and \textbf{Corollary \ref{CC2}}.

\hskip 1pt \\
 Chapter $C$ is different, we provide in this chapter  preliminaries to chapters $D, E,F$ and $G.$ 
\hskip 1pt \\
\hskip 1pt \\
In chapter $D,$ we consider $V,$ a reducible and
indecomposable $G-$module which has a unique  $G-$submodule $W$
with $dim_FW=2,$ and $V$ contains no $1-$dimensional $G-$submodules.In this chapter,we can not conclude that $S(V)^{T(G)}$ is always polynomial ring.So we make specific assumption (see \textbf{Assumption \ref{DA1}} )to prove [Theorem \ref{T3}].This is applicable in case $F=\mathbb{F}_p .$

\hskip 1pt \\
\hskip 1pt \\
In chapter $E,$ we consider $V,$ a reducible and indecomposable $G-module$ which has a unique
$1-$dimensional $G-$submodule $U=Fv_0$ and $V$ contains no
$2-$dimensional $G-$submodules.Hence  $W=V/U$ is an irreducible
$G-$ module. We exclusively consider here the case $F=\mathbb{F}_p.\\$
There are two cases of $W$ that  \textbf{Lemma \ref{EL1}} have provided .The first case is when  $W$ is an unfaithful irreducible $T(G)-$module, and the second one is when $W$ is a faithful irreducible $T(G)-$module.

\hskip 1pt \\
\hskip 1pt \\
In chapter $F,$ we consider $V,$ a reducible and
indecomposable $G-$module which has a unique  $G-$submodule $W$
with $dim_FW=2,$ and a unique $1-$dimensional $G-$submodule $Fv_0,$ and $Fv_0\subset W\subset V.$ We exclusively consider here the case $F=\mathbb{F}_p.$

\hskip 1pt \\
\hskip 1pt \\
Finally,in chapter $G,$ we consider $V,$ a reducible and indecomposable $G-module$ and contains (at least) two $1-$dimensional $G-$submodules.By using \textbf{Proposition \ref{FP1}} , we provide a sufficient and necessary condition on $G$ for the Gorenstein property of $S(V)^G,$ and  present the main result  in  \textbf{Corollary \ref{FF1}}.

}

\newpage

\addcontentsline{toc}{section}{Index of notations}

\begin{center}
{\Large
\textbf{\underline{Index of notations}}\\
\hskip 1pt \\
}
\end{center}

{\large
$\mathbb{F}_p-$ The prime field of order $p.$ \\\\
$\indent |G|-$ The order of group $G.$\\\\
$\indent F-$ A finite field with $char F=p.$ \\\\
$\indent \bar{F}-$ The algebraic closure of $F.$\\\\
$\indent S(V)-$ The symmetric algebra of $V.$\\\\
$\indent S(V)^G-$ The subring of $G-$ invariants.\\\\
$\indent T(G)-$ The $G-$subgroup generated by all transvections.\\\\
$\indent W(G)-$ The $G-$subgroup generated by all pseudoreflections.\\\\
$\indent m-$ The unique homogenous maximal ideal of $S(V)^{T(G)}.$\\\\
$\indent V^{*}-$ The  hom-dual $hom_F(V,F).$\\\\
$\indent GL(V)-$ The general linear group of $V.$\\\\
$\indent GL(n,F)-$ The linear automorphisms of $F^n.$\\\\
$\indent SL(V)-$ The special linear group of $V.$\\\\
$\indent SL(n,F)-$ The linear automorphisms of $F^n$ of determinant $1.$ \\\\
$\indent G|W-$ The restriction of $G$ to $W.$\\\\
$\indent <\sigma_1,...,\sigma_s>-$ The subgroup generated by $\{\sigma_1,...,\sigma_s\}.$\\\\
$\indent Fix_G(W)-$ The elements of $G,$ that acts trivially on subspace $W.$\\\\

}

\newpage
\pagenumbering{arabic}
\setcounter{page}{1}
\section{Introduction}

 $\indent$ Let $G\subset SL(V)$ be a finite subgroup 
of SL(V). Let $F$ be a field with $charF=p>0$ , $V$ a $3-$
dimensional $F-$ vector space  and $S(V)$ the
symmetric algebra of $V$ .We denote by $S(V)^G$ the subring of
$G-invariants$. One of the main purposes of this thesis is to
classify Gorenstein rings
of the form ${S(V)}^G,$ when $p$ divides $|G|.$\\ \\
The following are results of K.Watanabe ,Fleischmann-Woodcock and A.Braun . \\

\begin{theorem}\label{T1}(K.Watanabe \cite{W1},\cite{W2})\\
Suppose that $(|G|,p)=1$  and $G$ contains no
pseudo-reflection.Then the following are equivalent:
\begin{enumerate}
\item $G\subset SL(V).$ \item $S(V)^G$ is a Gorenstein ring.

\end{enumerate}
\end{theorem}

\begin{theorem}\label{T2}(A.Braun \cite{Braun2},Fleischmann-Woodcock \cite{FW})\\
Suppose that  $S(V)^{G}$   is Cohen-Macaulay and $S(V)^{W(G)}$ is
a polynomial ring. Then $S(V)^{G}$ is Gorenstein iff
$G/W(G)\subseteq SL(m/m^{2}),$ where $m$ is the unique homogenous
maximal ideal of $S(V)^{W(G)}$ and $W(G)$ is the $G-subgroup$
generated by all pseudo-reflections(of all types).

\end{theorem}

Here an element $g \in G $ is called pseudoreflection if
$rank(g-I)=1$.
  If $(|G|,p)=1$ then every pseudoreflection is diagonalizable with one's on all but a single  diagonal entry.
 In the modular case, that is if $ p $  divides $|G|$ we have a new type of pseudoreflection, called transvection,
  having the additional property $(g-I)^2=0,$ so it is not
  diagonalizable.Denote by $T(G)$ the $G-subgroup$ generated by
  all transvections .
  \\ \\
$\indent$  Eagon and Hochster proved in \cite{EH} that $S(V)^G$ is
Cohen-Macaulay for all nonmodular groups $G,$ i.e when
$(|G|,p)=1.$ But there are finite $p-groups,$ such that $S(V)^G$ is not Cohen-Macaulay.\\

\begin{example}
 Let $F$ be a field of characteristic $p$ and
$V$ is a faithful representation of a non-trivial $p$-group $P$
and $mV$ denote the faithful representation of $G$ formed by
taking the direct sum of $m$ copies of $V.$ Then $S(mV)^P$ is not
Cohen-Maculay when $m\geq 3.$\cite [Theorem 1.2]{CG}
\end{example}

However in this thesis we assume that $dim_FV=3$ and by \cite [Proposition 5.6.10]{NS}
$S(V)^G$ is always Cohen-Macaulay ring. 

\newpage
It is shown in \cite [Section 2]{Braun3} that $S(V)^G$ can fail to be Gorenstein even if $G\subset SL(V)$ and $dim V=2.$ Moreover, \cite{Braun3} provides necessary and sufficient conditions to ensure the Gorenstein property of $S(V)^G$ in case $dim_FV=2$ and $G\subset SL(V).$ We mostly examine here the case of $dim_FV=3, G\subset SL(V)$ and $V$ is a reducible $G$-module. The analysis is carried out by separation into different cases. A noteworthy consequence is perhaps the following:
\begin{theorem}\label{T3}
Let $G\subset SL(3,\mathbb{F}_p)$ be a finite group and $V=\mathbb{F}_p^3$ is a reducible $G$-module.
Then $S(V)^G$ is Gorenstein. 
\end{theorem}
$\\\\$
 This extends a similar result in \cite [Corollary (C)]{Braun3}.The separation into cases is summed up in the following table.\\\\

\begin{center}
\begin{tabular}{ |p{2cm}|p{10cm}|  }

\hline
\multicolumn{2}{|c|}{\textbf{Table}} \\
\hline
\textbf{Chapter}& \textbf{$V$ is a reducible $G$-module.} \\
\hline
$A$ & $V=Fv\oplus W$ and $Fv,W$ are $G$-submodules.   \\ \hline
$B$ & $V$ is  indecomposable and contains (at least) two $2-$ dimensional $G-$ submodules $W_1,W_2.$    \\ \hline
$D$ & $V$ is  indecomposable,contains a unique $2-$ dimensional $G-$ submodule, but no
$1-$ dimensional $G-$ submodule.\\ \hline
$E$  &$V$ is  indecomposable,contains a unique $1-$ dimensional $G-$ submodule $U,$ but no
$2-$ dimensional $G-$ submodule.\\ \hline
$F$  &$V$ is  indecomposable and $\mathbb{F}_pv_0\subset W\subset V,$where $\mathbb{F}_pv_0,W$ are $G-$ submodules and  $dim_{\mathbb{F}_p}W=2.$ 
\\ \hline
$G$  &$V$ is  indecomposable and contains (at least) two $1-$ dimensional $G-$ submodules. 
\\ \hline
\end{tabular}

\end{center}
$\\\\\\\\$
\textbf{\underline{Convention:}} Given a basis $\{v_1,...,v_n\}$ of $V,$ and $g\in G.$
Then $g(v_i)=\sum_{j=1}^{n
    }a_{ij}v_j,a_{ij}\in F $ and the matrix $A:=(a_{ij}),$ representing $g,$ acts on coordinates (=row vectors) from the \underline{right} .
\newpage

\section{Chapter $(A):V=Fv\oplus W,$ and $Fv,W$ are $G$-submodules. }

$\indent$In this chapter we consider  $V ,$  a reducible 
 $G-module$ of the form  $V=Fv
\oplus W$  and $Fv$ , $W$  are $G-modules$ with $dim_{F}W=2$ . 
\textbf{Proposition \ref{AP1}}  provides a sufficient and necessary
condition for  $G$ as above,to ensure the Gorenstein property of $S(V)^G.$  Our proof will proceed in steps by dividing it
into $3$ cases:

\begin{enumerate}
\item $W$ is an irreducible and primitive $T(G)-module.$\\

\item $W$ is a reducible $T(G)-module.$\\
\item $W$ is an irreducible and imprimitive $T(G)-module.$\\

\end{enumerate}

The next two  results are usefull  in this chapter.
\begin{lemma}\label{AL1}
     $W$ is a faithful $G$-module.

     \end{lemma}

        \begin{proof}

        Let $1_G\neq g\in G$ , with $g|W=Id$ .Then \\

         \[ g= \left( \begin{array}{ccc}
\lambda(g)& 0 & 0 \\
0 & 1& 0 \\
0 & 0 & 1 \end{array} \right).\]

Where $g(v)=\lambda(g)v$ , for all $g\in G$ .Since $G\subset
SL(V)$ , $\lambda(g)=1$ and therefore
     \[ g= \left( \begin{array}{ccc}
1& 0 & 0 \\
0 & 1& 0 \\
0 & 0 & 1 \end{array} \right).\] A contradiction.

\end{proof}

The next result does not require the assumption $V=Fv\oplus W.$

\begin{lemma}\label{AL2}
Suppose that $G\subset SL(V)$ and $U$ a $1$-dimensional
$G-submodule.$ Then $U\subseteq ker(\sigma-I)$ for all
transvection $\sigma $ in $G.$ Equivalently, $T(G)$ acts trivially on
$U.$
\end{lemma}

\begin{proof}
  Let $\sigma \in T(G)$ be a transvection ,$\sigma\neq I ,$ and
$M=ker(\sigma-I).$ Suppose that $Fu=U\not\subset M.$ Then with respect to a basis $\{u,m_1,m_2\},m_1,m_2\in M,$ we have :
 $\sigma=\left(\begin{matrix}
\lambda(\sigma) &0 & 0 \\
0 & 1 & 0 \\
0 & 0 & 1
\end{matrix} \right),$
where $\sigma(u)=\lambda(\sigma)u$ and
$\sigma|M=Id.$ But $G\subset SL(V),$ therefore $\lambda(\sigma)=1 ,$
implying that $\sigma=I,$ a contardiction.

\end{proof}

\newpage

\begin{proposition}\label{AP1}
   Let $G\subset SL(V)$ be a finite group and $V=Fv\oplus W$ a
decomposition of $V$ into $G-submodules$ with $dim_FW=2$.Then the
following are equivalent:
\begin{enumerate}
 \item $S(V)^G$ is Gorenstein.

 \item $det(g|W)=det(g|{m/m^2}),$ for each $g\in G$  where $m$ is the unique
 homogenous maximal ideal of $S(W)^{T(G)}$.
\end{enumerate}
\end{proposition}
\begin{proof}
 By \textbf{{Lemma } \ref{AL2}} ${S(V)}^{T(G)}={S(W)}^{T(G)}[v]$.Now by
\cite{Nakajima}, $S(W)^{T(G)}=F[a_1,a_2]$ is a polynomial ring,where
$m=(a_1,a_2)$. Let $P$ be the unique homogenous maximal ideal of
$S(V)^{T(G)},$ then $P=(a_1,a_2,v).$ By [Theorem \ref{T2}]
$S(V)^{G}$ is Gorenstein if and only if $1=det(g|P/P^{2})=
\lambda(g)det(g|m/m^{2})$ , where $g(v)=\lambda(g)v$ , for all $g
\in G. $ Since $G\subset SL(V),det(g|W)={\lambda(g)}^{-1},$ hence
$S(V)^{G}$ is Gorenstein if and only if $det(g|W)=det(g|{m/m^2}),$ for each $g\in G.$

\end{proof}

 The transvections group $T(G)|W$  had been classified  by Kantor
whenever $T(G)|W$ is an irreducible primitive linear group
\cite[Theorem 1.5]{Kemper3}.
So with the above notation one of the following holds: \\
\begin{enumerate}[label=(\roman*)]
\item $T(G)|W=SL(2,\mathbb{F}_q),$ where $p|q.$\\
 \item $T(G)|W\cong SL(2,\mathbb{F}_5),T(G)|W\subset SL(2,\mathbb{F}_9)$ and $\mathbb{F}_9\subseteq F.\\$

 \end{enumerate}
 We now consider the above case \textbf{(i)}. \\
 \begin{proposition}\label{APb}

Suppose $G\subset SL(V)$ is a finite group with
$T(G)|W=SL(2,\mathbb{F}_q)$ ,where $q=p^s$ and
$\mathbb{F}_q\subseteq F.$ Let $V=F^3=Fv\oplus W $ be the natural
$SL(3,F)-module .$ Then $S(V)^{G}$ is Gorenstein if and only if $G|W \subset
GL(2,\mathbb{F}_{q^2}).$

\end{proposition}\label{AP2}
\begin{proof}

Recall  as in \cite [Proposition 3.5] {Braun3} that $S(W)^{T(G)}=S(W)^{SL(2,\mathbb{F}_q)}=F[u,c_{21}],deg(u)=q+1,deg(c_{21})=q^2-q.$
Let $g\in G$ be arbitrary and set $g|W=\bigl(\begin{smallmatrix} a&b \\
c&d
\end{smallmatrix} \bigr)$ with respect to the basis $\{x,y\}.$ So
$g^{-1}|W=\frac{1}{det(g|W)}\bigl(\begin{smallmatrix} d&-b \\
-c&a
\end{smallmatrix} \bigr).$Since $gT(G)g^{-1}=T(G),$ we get after
restriction to $W$ that $(g|W) \bigl(\begin{smallmatrix} 1&1 \\
0&1
\end{smallmatrix} \bigr)(g^{-1}|W)=\frac{1}{det(g|W)}\bigl(\begin{smallmatrix} a&b \\
c&d
\end{smallmatrix} \bigr)\bigl(\begin{smallmatrix} 1&1 \\
0&1
\end{smallmatrix} \bigr)\bigl(\begin{smallmatrix} d&-b \\
-c&a
\end{smallmatrix}\bigr)=\frac{1}{det(g|W)}\bigl(\begin{smallmatrix} ad-bc-ac&a^2 \\
-c^2&ad-bc+ac
\end{smallmatrix} \bigr)$ is in $SL(2,\mathbb{F}_q).$\\\\Similarly,
$(g|W)\bigl(\begin{smallmatrix} 1&0 \\
1&1
\end{smallmatrix} \bigr)(g^{-1}|W)=\frac{1}{det(g|W)}\bigl(\begin{smallmatrix} ad-bc+bd&-b^2 \\
d^2&ad-bc-bd
\end{smallmatrix} \bigr)\in SL(2, \mathbb{F}_q).$Set
$e=det(g|W)=ad-bc.$ Then
$\frac{ac}{e},\frac{a^2}{e},\frac{c^2}{e},\frac{b^2}{e},\frac{d^2}{e},\frac{bd}{e}\in
\mathbb{F}_q.$ Assume firstly that $a\neq 0.$Then
$\gamma:=\frac{c}{a}=\frac{(\frac{c^2}{e})}{(\frac{ac}{e})}\in
\mathbb{F}_q.$  If $b\neq 0$ then
$\mu:=\frac{d}{b}=\frac{(\frac{bd}{e})}{(\frac{b^2}{e})}\in
\mathbb{F}_q.$Hence $e=ad-bc=ab\mu-ba\gamma=ab(\mu-\gamma)$ and
therefore
$\beta:=\frac{b}{a}=\frac{ab(\mu-\gamma)}{a^2(\mu-\gamma)}=\frac{e}{a^2(\mu-\gamma)}\in
\mathbb{F}_q,$ as well as $\delta:=\frac{d}{a}=\frac{\mu b}{a}=\mu
\beta \in \mathbb{F}_q.$If $b=0$ then take $\beta=0$ and
$\delta:=\frac{d}{a}=\frac{ad}{a^2}=\frac{e}{a^2}\in
\mathbb{F}_q.$So in both cases $b=a\beta,c=a\gamma,d=a\delta,$
$\beta,\gamma,\delta\in
\mathbb{F}_q.$ Therefore $g|W=\bigl(\begin{smallmatrix} a&0 \\
0&a
\end{smallmatrix}\bigr)\bigl(\begin{smallmatrix} 1&\beta \\
\gamma&\delta
\end{smallmatrix}\bigr),$ with $h:=\bigl(\begin{smallmatrix} 1&\beta \\
\gamma&\delta
\end{smallmatrix}\bigr)\in GL(2,\mathbb{F}_q).$
\newpage
If $a=0$ then $e=det(g|W)=-bc$ implies $b\neq 0 \neq c.$ So we get
equality $g|W=\bigl(\begin{smallmatrix} c&0 \\
0&c
\end{smallmatrix}\bigr)\bigl(\begin{smallmatrix} 0&\beta \\
1&\delta
\end{smallmatrix}\bigr)$ with $c$ replacing $a$ and $\beta,\delta\in \mathbb{F}_q.$Consequently if $a\neq 0$ we get: \\
$g(c_{21})=(aI_W)(h(c_{21}))=a^{q^2-q}c_{21},$ and as in the proof
of \cite[Proposition 3.5] {Braun3} we get that
$g(u)=(aI_W)(h(u))=det(h)a^{degu}u=det(h)a^{q+1}u.$ But
$det(g|W)=a^2det(h),$ hence $g(u)=det(g|W)a^{q-1}u.$\\\\ Let
$M=Fu+Fc_{21}$ be the $2-$dimensional subspace of the polynomial
ring $S(W)^{T(G)}=F[u,c_{21}],m=(u,c_{21}),$ so $m/m^2\cong M.$ Then the
matrix representing $g|M$ with respect to the basis $\{u,c_{21}\}$
is $\bigl(\begin{smallmatrix} a^{q-1}det(g|W)&0 \\
0&a^{q^2-q}
\end{smallmatrix}\bigr).$
Therefore the condition of \textbf{Proposition \ref{AP1} }:
$det(g|W)=det(g|m/m^2)=det(g|M)$ is equivalent to:
$det(g|W)=a^{q-1}det(g|W)a^{q^2-q}=a^{q^2-1}det(g|W).$ Hence it is
equivalent to $a^{q^2-1}=1,$ namely $a\in \mathbb{F}_{q^{2}}.$ Using $h\in GL(2,\mathbb{F}_{q}),$
 this is also equivalent to $g|W=\bigl(\begin{smallmatrix} a&b
\\c&d
\end{smallmatrix}\bigr)\in GL(2,\mathbb{F}_{q^{2}}).$ A similar
conclusion is obtained if $a=0 .$

\end{proof}

$\\$
We next consider the above case \textbf{(ii)}. \\

\begin{proposition}\label{AP3}
Suppose $G\subset SL(3,F)$ is a finite group with $T(G)|W\cong
SL(2,\mathbb{F}_5),T(G)|W\subseteq SL(2,\mathbb{F}_9)$ and
$\mathbb{F}_9\subseteq F$ (here $p=3$).Let $V=F^3=Fv\oplus W $ be the natural
$SL(3,F)-module.$ Then $S(V)^G$ is Gorenstein iff $G|W\subseteq
<T(G)|W,\eta I_W>,$ where $\eta$ is a $20^{th}$-primitive root of unity.

\end{proposition}\label{AP4}

\begin{proof}

Let $g\in G ,$ choose $\xi\in \bar{F} $ such that
$\xi^{2}det(g|W)=1.$ Hence $\hat{g}:=\xi I_Wg|W\in SL(2,\bar{F})$
where $\bar{F}$ is the algebraic closure of $F$.Clearly
$\hat{g}(T(G)|W)\hat{g}^{-1}=T(G)|W.$ Therefore the computation in
\cite[Proposition 3.7]{Braun3}, with $\hat{g}$ instead of
$g|W,$ shows that $\hat{g}\in T(G)|W.$ Hence $\bigl(\begin{smallmatrix} \xi&0 \\
0&\xi
\end{smallmatrix} \bigr)=\hat{g}({g}^{-1}|W)\in
(T(G)|W)(G|W)=G|W.$ Consequently $G|W\subseteq <T(G)|W,(\begin{smallmatrix} \lambda&0 \\
0&\lambda
\end{smallmatrix} \bigr)|\lambda\in \bar{F}>.\\$

Now $\bigl(\begin{smallmatrix} \xi&0 \\
0&\xi
\end{smallmatrix} \bigr)g|W=\hat{g}|W\in T(G)|W.$ Recall from \cite[Proposition 6.6.]{Braun3} that $S(W)^{SL(2,\mathbb{F}_5)}=F[f_{10},f_{12}],$ where
$f_{10}=x_1^9x_2-x_1x_2^9$ and $f_{12}=x_1^{12}+x_1^{10}x_2^2-x_1^6x_2^6+x_1^2x_2^{10}-x_2^{12}.$
Hence
$\hat{g}(f_{10})=f_{10},\hat{g}(f_{12})=f_{12}$ implying that
$g(f_{10})=\xi^{-10}f_{10},g(f_{12})=\xi^{-12}f_{12}.\\$ Let 
$U=Ff_{10}+Ff_{12}$ be the $2-$dimensional subspace of the polynomial
ring $S(W)^{SL(2,\mathbb{F}_5)}=F[f_{10},f_{12}],m=(f_{10},f_{12}),$  so $m/m^2\cong U.$ Then the
matrix representing $g|U$ with respect to the basis $\{f_{10},f_{12}\}$
is $\bigl(\begin{smallmatrix} \xi^{-10}&0 \\
0&\xi^{-12}
\end{smallmatrix} \bigr).$ Therefore, by \textbf{Proposition \ref{AP1} }
$det(g|W)=det(g|m/m^{2})=det(g|U)$ is translated into
$\xi^{-2}=\xi^{-22}$ or $\xi^{20}=1.$ Therefore $S(V)^G$ is
Gorenstein iff $G|W\subseteq<T(G)|W,\eta I_W>$ where $\eta\in
\bar{F}$ is a primitive $20-th$ root of unity.

\end{proof}

We now consider case $2,$ namely   the possibility of $W$ being a reducible
$T(G)-module.$

\begin{proposition}\label{AP5}
Let $G\subset SL(3,F)$ be a finite group.Let $V=F^3=Fv\oplus W $ be the
natural $SL(3,F)-module$ and suppose $W$ is a \textbf{reducible}
$T(G)$-submodule.Then $S(V)^G$ is Gorenstein if and only if
$G|W\subseteq \{\bigl(\begin{smallmatrix} a&b\\ 0&d
\end{smallmatrix} \bigr) \vert a^{p^n-1}=1, a,b,d\in F\}.$

\end{proposition}

\begin{proof}
Set $\hat{g}:=\xi I_Wg|W ,$ with $det(\hat{g})=\xi^{2}det(g|W)=1.$
Since $W$ is a reducible $T(G)|W-module,$there exists a basis
$\{x_1,x_2\}$ of $W$ with respect to which $T(G)|W\subseteq
\bigl(\begin{smallmatrix} F&F\\ 0&F
\end{smallmatrix} \bigr).$ Every transvection  acts trivially on every
one dimensional submodule of $W$ by \textbf{Lemma \ref{AL2}}. Hence it
must stabilize $Fx_2.$ Therefore since $T(G)|W\subset SL(W),$
$T(G)|W\subseteq \bigl(\begin{smallmatrix} 1&F\\ 0&1
\end{smallmatrix} \bigr).$ Consequently $T(G)|W$ is elementary abelian group.So
$T(G)|W=<\sigma_1>\times\cdot \cdot \cdot\times<\sigma_n>,$ where
$\sigma_i=\bigl(\begin{smallmatrix} 1&\alpha_i\\ 0&1
\end{smallmatrix} \bigr),$ for $i=1,...,n,$ and
$\{\alpha_1,...,\alpha_n\}$ are linearly independent over
$\mathbb{F}_p.$Let  $g|W=\bigl(\begin{smallmatrix} a&b \\
c&d
\end{smallmatrix} \bigr).$ Now the computation in the proof of \cite[Theorem 3.1]{Braun3} with $\hat{g}$ instead of $g|W$ shows that $\xi c=0 ,$ hence
$c=0 ,$
 $\hat{g}=\bigl(\begin{smallmatrix} \lambda&\pi\\
0&\lambda^{-1}
\end{smallmatrix} \bigr).\\$ Hence $\hat{g}|m/m^2=\bigl(\begin{smallmatrix} \lambda^{p^n}&0\\
0&\lambda^{-1}
\end{smallmatrix} \bigr),$ where $\xi a=\lambda,$ $\xi d=\lambda^{-1},\pi\in F$ and $m=(b,x_2) ,$ where $b$ which is of degree $p^n$ is defined as in the proof of
\cite[Theorem 3.1]{Braun3}.Hence $\hat{g}(b)=\lambda^{p^n}b=(\xi
a)^{p^n}b=(\xi)^{p^n}g(b)$ and $\hat{g}(x_2)=\lambda^{-1}x_2=(\xi
d)x_2=\xi g(x_2).$ Therefore $g|m/m^2=\bigl(\begin{smallmatrix} a^{p^n}&0\\
0&d\end{smallmatrix} \bigr).$ Consequently
$det(g|{m/m^2})=a^{p^n}d=a^{p^n-1}(ad)=a^{p^n-1}det(g|W)$.Hence, by
\textbf{Proposition \ref{AP1} } $S(V)^G$ is Gorenstein if and only if
$a^{p^n-1}=1$.In other words $S(V)^G$ is Gorenstein if and only if
$G|W\subseteq \{\bigl(\begin{smallmatrix} a&b\\ 0&d
\end{smallmatrix} \bigr) \vert a^{p^n-1}=1 , b,d\in F\}.$

\end{proof}

         Our next objective is to handle case $3,$ i.e when $T(G)|W$ acts irreducibly and imprimitively on $W.$
It is well known(e.g.\cite{ZS}) that $T(G)|W$ is a monomial
subgroup.Recall that $H\subset GL(W)$ is called monomial if $W$
    has a basis with respect to which the matrix of each element
    of $H$ has exactly one non-zero entry in each row and
    column. If $charF\neq 2$ , then by \cite[Lemma 3.9.]{Braun3} $T(G)|W={1} .$This contradicts  the assumption that $W$ is an irreducible
    $T(G)|W-$module.Hence we only deal with $p=2.$

\begin{proposition}\label{AP6}
Assume that $charF=2.$ Let $G\subset SL(3,F)$ be a finite group
  with $T(G)|W$ acting imprimitively and
irreducibly on the  submodule $W$ and $V=Fv\oplus W.$  Then $S(V)^G$ is
Gorenstein iff $G|W\subseteq<T(G)|W,\delta I_W>,$ where $\delta$
is a primitive $d-th$ root of unity and $d$ defined in the proof.

\end{proposition}
\begin{proof}
Let $\{x_1,x_2\}$ be the basis of $W$ with respect to which
$T(G)|W$ is irreducible and monomial.Let $\{g_1,...,g_n\}$ be the
set of transvections generating $T(G)|W.$The $T(G)|W-$
irreducibility of $W$ forces $n\ge 2.$ So either
$g_i=\bigl(\begin{smallmatrix} *&0 \\ 0 &*
\end{smallmatrix} \bigr)$ or $\bigl(\begin{smallmatrix} 0&*\\ * &0
\end{smallmatrix} \bigr).$ The first possibility leads , as in \cite[Lemma 3.9.]{Braun3} to $g_i|W=I_W.$ So  $g_i=\bigl(\begin{smallmatrix} 0&\alpha_i \\
\alpha_i^{-1}&0
\end{smallmatrix} \bigr),$ for $i=1,...,n,$ with respect to the basis $\{x_1,x_2\}$
of $W.$ Let $n_{ij}$ be the minimal number such that $(g_ig_j)^{n_{ij}}=1,$where
$i\neq j,$ clearly $g_i^{-1}=g_i$ for $i=1,...,n.$ Also
$n_{ij}=n_{ji}>1$ and  since $charF=2,$ $n_{ij}$ is odd. Recall
that $S(W)^{T(G)}$ is a polynomial ring \cite[Theorem 2.4]{Nakajima}.We next compute the actual generators of $S(W)^{T(G)}.$ \\

Let $[(T(G)|W),(T(G)|W)]=<g_ig_jg_ig_j>$ be the commutator subgroup of
$T(G)|W.$
 We have $(g_1g_2)^{n_{12}}=1.$ 
Therefore
$g_2=(g_1g_2)^{n_{12}}g_2=(g_1g_2)^{n_{12}-1}(g_1g_2)g_2=(g_1g_2)^{n_{21}-1}g_1.$
But $n_{12}-1$ is even, since $n_{12}$ is odd .Hence
$g_2=(g_1g_2g_1g_2)^{\frac{n_{12}-1}{2}}g_1\in
[T(G)|W,T(G)|W]g_1.$ This similarly holds for $g_i, i\geq 3,$ hence $\frac
{|(T(G)|W)|}{|([T(G)|W,T(G)|W])|}=2.$
Let $d=l.c.m\{n_{ij}|i<j\}.$ Since $g_ig_jg_ig_j=\bigl(\begin{smallmatrix} \alpha_i\alpha_j^{-1}\alpha_i\alpha_j^{-1}&0 \\
0&\alpha_i^{-1}\alpha_j\alpha_i^{-1}\alpha_j
\end{smallmatrix} \bigr),$ it follows that
$[(T(G)|W),(T(G)|W)]=\{ \bigl(\begin{smallmatrix} \zeta&0 \\
0&\zeta^{-1}
\end{smallmatrix} \bigr)\vert \zeta^d=1 \}.$  Since
$(\alpha_i\alpha_j^{-1})^{n_{ij}}=1$ it follows that
$\alpha_i^{n_{ij}}=\alpha_j^{n_{ij}}.$ Consequently
$\alpha_i^d:=\beta$  for $i = \{1,2,3,..,n\}.$
\\\\

 Now we have:
\begin{enumerate}
\item
$g_i(x_1^d+\alpha_i^dx_2^d)=g_i(x_1)^d+\alpha_i^dg_i(x_2)^d=\alpha_i^dx_2^d+\alpha_i^d(\alpha_i^{-1}x_1)^d=x_1^d+\alpha_i^dx_2^d$.
\item $g_i(x_1x_2)=g_i(x_1)g_i(x_2)=\alpha_i x_2\cdot
\alpha_i^{-1}x_1=x_1x_2$.
\end{enumerate}

Hence $\{x_1x_2,x_1^d+\beta x_2^d\}\in S(W)^{T(G)} $ and 
$deg(x_1x_2)deg(x_1^d+\beta x_2^d)=2d=|T(G)|W|.$ So by \cite[Theorem
3.7.5]{Kemper1} we get that $S(W)^{T(G)}=F[x_1x_2,x_1^d+\beta x_2^d].$
Let $g|W=\bigl(\begin{smallmatrix} a&b \\ c&d
\end{smallmatrix} \bigr).$ Set  $\hat{g}=\xi I_Wg|W ,$ where
$\xi^2det(g|W)=1.$ Now the computation of \cite[Theorem 3.10]{Braun3},with
$\hat{g}$ instead of $g|W,$ shows that $\hat{g}\in T(G)|W.$
\newpage
 We have:
\begin{enumerate}[label=(\roman*)]
\item $g(x_1x_2)=\bigl(\begin{smallmatrix} \xi^{-1}&0 \\
0&\xi^{-1}
\end{smallmatrix} \bigr)\hat{g}(x_1x_2)=\xi^{-2}x_1x_2.$ 

\item $g(x_1^d+\beta x_2^d)=\bigl(\begin{smallmatrix} \xi^{-1}&0 \\
0&\xi^{-1}
\end{smallmatrix} \bigr)\hat{g}(x_1^d+\beta x_2^d)=\bigl(\begin{smallmatrix} \xi^{-1}&0 \\ 0&\xi^{-1}
\end{smallmatrix} \bigr)(x_1^d+\beta x_2^d)=\xi^{-d}(x_1^d+\beta
x_2^d).$ \end{enumerate}

Consequently, $g|{m/m^2}=\bigl(\begin{smallmatrix} \xi^{-2}&0 \\
0&\xi^{-d}
\end{smallmatrix} \bigr)$ ,where $m/m^2=U=F(x_1x_2)+F(x_1^d+\beta
x_2^d).$ So $det(g|{m/m^2})=\xi^{-2}\xi^{-d}=det(g|W)\xi^{-d},$
and $det(g|W)=det(g|{m/m^2})$ if and only if $\xi^{-d}=1.$
Consequently, by \textbf{Proposition \ref{AP1} } $S(V)^G$ is Gorenstein
if and only if $G|W\subseteq<T(G)|W,\delta I_W>$ where $\delta$ is a
primitive $d-th$ root of unity.

\end{proof}

\begin{corollary}\label{AC1}
Suppose $G\subset SL(3,\mathbb{F}_p),$ and $V=\mathbb{F}_pv
\oplus W$ a decomposition of $V=\mathbb{F}_p^{3}$ into $G-$submodules.Then $S(V)^{G}$ is
Gorenstein.

\end{corollary}
\begin{proof}
This is a consequence of \textbf{Proposition \ref{APb} }, \textbf{Proposition \ref{AP5} } and  \textbf{Proposition \ref{AP6} }.
\end{proof}

\newpage

\section{Chapter($B$): $V$ is  indecomposable and contains (at least) two $2-$ dimensional $G-$ submodules $W_1,W_2.$ }

$\indent$In this chapter we consider $V,$ a reducible,
indecomposable $G-module$ which has at least two $G$-submodules
$W_1,W_2$ where $dimW_i=2$ for $i=1,2.$ Fix a basis
$\{w_2,w_1,v\}$ such that $W_i=Fw_i+Fv.$ Then $W_1\cap W_2=Fv$ is
$1$-dimensional $G$-submodule of $V.$ Therefore for every $g\in G$
we have a matrix representation  with respect to  this basis.
\[ g=\left(
\begin{array}{ccc}
\lambda_2(g)& 0 & \delta_{13}(g) \\
0 & \lambda_1(g)& \delta_{23}(g) \\
0 & 0 & \lambda(g) \end{array} \right).\] Let $\sigma\in T(G)$ be
a transvection . $\sigma:W_i\to W_i$ for $i=1,2.$ So
$\sigma(w_i)=\alpha_iw_i+\beta_iv$ where $\alpha_i,\beta_i \in F.$
Since $|<\sigma>|=p$ and $\sigma$ acts trivially on $Fv $ by
\textbf{Lemma \ref{AL2}}, $w_i=\sigma^{p}(w_i)=\alpha^{p}_iw_i+\gamma_i v
,$ hence $\gamma_i=0$ and $\alpha_i^{p}=1 ,$ then  $\alpha_i=1,$ for $i=1,2.$ As a
result, $\sigma$ is represented by the matrix :
\[  \left(
\begin{array}{ccc}
1& 0 & a \\
0 & 1& b \\
0 & 0 & 1\end{array} \right)\] with respect the basis
$\{w_2,w_1,v\}.$

\begin{corollary}\label{CC1}
$T(G)$ is an elementary abelian group.
\end{corollary}

\begin{proof}
Suppose that $\tau$ is another transvection where \[ \tau= \left(
\begin{array}{ccc}
1& 0 & a_1 \\
0 & 1& b_1 \\
0 & 0 & 1\end{array} \right).\] Then it easy to see that $\tau
\sigma=\sigma \tau.$ Therefore $T(G)$ is an elementary abelian
group.Say it is minimaly  generated by $\{\sigma_1,...,\sigma_n\},$ where
\[ \sigma_i= \left(
\begin{array}{ccc}
1& 0 & a_i \\
0 & 1& b_i  \\
0 & 0 & 1\end{array} \right)\]for $i=1,...,n.$ So $|T(G)|=p^n.$

\end{proof}

 Let  $X:=\{\bigl(\begin{smallmatrix} a_i
\\ b_i
\end{smallmatrix}\bigr)\in (\begin{smallmatrix} F
\\ F
\end{smallmatrix}\bigr)|i=1,...,n\},$ then the elements of $X$ are linearly independent  over
$\mathbb{F}_p.$ Clearly $dim (span_F X)\in \{1,2\}.$

\newpage

\subsection*{Case 1: $dim (span_F X)=1 $ }
$\indent$ Assume that $dim (span_F X)=1.$So
$\bigl(\begin{smallmatrix} a_i
\\ b_i
\end{smallmatrix}\bigr)=\mu_i\bigl(\begin{smallmatrix} a
\\ b
\end{smallmatrix}\bigr)$ for $i=1,...,n.$ Where $\bigl(\begin{smallmatrix} a_1
\\ b_1
\end{smallmatrix}\bigr)=\bigl(\begin{smallmatrix} a
\\ b
\end{smallmatrix}\bigr), \mu_1=1$ and $\mu_i\in F$
for $i=2,...,n.$ Since the elements of $X$ are linearly
independent over $\mathbb{F}_p,$ this is equivalent to
$\{1,\mu_2,...,\mu_n\}$ being of linearly independent
 over $\mathbb{F}_p.$

\begin{lemma}\label{CL1}

$S(V)^{T(G)}$ is a polynomial ring.
\end{lemma}
\begin{proof}

We shall firstly change the basis of $V.$ We have
$\sigma_1(w_1)=w_1+bv$ and $\sigma_1(w_2)=w_2+av,$ where $a=a_1$
and $b=b_1.$Hence
$\sigma_1(aw_1-bw_2)=a\sigma(w_1)-b\sigma(w_2)=a(w_1+bv)-b(w_2+av)=aw_1-bw_2.\\$
So $0\neq(aw_1-bw_2)\in S(V)^{<\sigma_1>}.$We therefore change the
basis of $V$ into $\{w_2,aw_1-bw_2,v\},$ and then $\sigma_1$ is
represented by the matrix: \[ \left(
\begin{array}{ccc}
1& 0 & a \\
0 & 1& 0 \\
0 & 0 & 1\end{array} \right).\]

We also have:\\

$\sigma_i(aw_1-bw_2)=a\sigma_i(w_1)-b\sigma_i(w_2)=a(w_1+b_iv)-b(w_2+av_i)=aw_1-bw_2+(ab_i-ba_i)v.$
But $\bigl(\begin{smallmatrix} a_i
\\ b_i
\end{smallmatrix}\bigr)=\mu_i\bigl(\begin{smallmatrix} a
\\ b
\end{smallmatrix}\bigr),$ and therefore $(ab_i-ba_i)=0.$ Hence
$0\neq(aw_1-bw_2)\in S(V)^{<\sigma_i>},$ for $i=1,...,n.$ So each
$\sigma_i$ is represented by the matrix:
\[ \left(
\begin{array}{ccc}
1& 0 & a_i \\
0 & 1& 0 \\
0 & 0 & 1\end{array} \right)\] where $a_i=\mu_i a.$ So we are in
the same situation as in the modular binary case
(see \cite{Braun3}).Consequently, by \cite[Lemma 3.13]{Braun3}
$S(V)^{T(G)}=F[x,aw_1-bw_2,v]$ is a polynomial ring, and $x$ is
an homogenous polynomial in $\{v,aw_1-bw_2\}$ of degree $p^n$ having
$w_2^{p^n}$ as one of its monomial.

\end{proof}

We now need to shift to the action of $G$ on $S(V)^{T(G)}.$
\\

\begin{proposition}\label{CP1}

$S(V)^G$ is Gorenstein if and only if  every element of $G/T(G)$
has order dividing $|T(G)|-1=p^n-1.$

\end{proposition}

\begin{proof}
If $a=0$ or $b=0$ we argue exactly as in the binary form paper
(see\cite{Braun3}).So we may assume that $a\neq 0$ and $b\neq 0.$ Since
$T(G)\lhd G,g\sigma_ig^{-1}\in T(G), \forall i=1,...,n.$ Hence
\[g\sigma_ig^{-1}= \varphi_i=\left(
\begin{array}{ccc}
1& 0 & \epsilon_i a\\
0 & 1& \epsilon_i b \\
0 & 0 & 1\end{array} \right),\epsilon_i\in \mathbb{F}_p1
+\mathbb{F}_p\mu_2...+\mathbb{F}_p\mu_n. \] So
$g\sigma_i=\varphi_i g$ is written explicitly by:
 \[\left(
\begin{array}{ccc}
\lambda_2(g)& 0 & \delta_{13}(g) \\
0 & \lambda_1(g)& \delta_{23}(g) \\
0 & 0 & \lambda(g) \end{array} \right)\left(
\begin{array}{ccc}
1& 0 & \mu_i a\\
0 & 1& \mu_i b \\
0 & 0 & 1\end{array} \right)=\left(
\begin{array}{ccc}
1& 0 & \epsilon_i a\\
0 & 1& \epsilon_i b \\
0 & 0 & 1\end{array} \right)\left(
\begin{array}{ccc}
\lambda_2(g)& 0 & \delta_{13}(g) \\
0 & \lambda_1(g)& \delta_{23}(g) \\
0 & 0 & \lambda(g) \end{array} \right).\] 

$$\\$$
By equating the $(1,3)$ and $(2,3)$ entries in both sides , we get :
 \[\left\{
\begin{array}{ll}
  \lambda_2(g)\mu_ia=\epsilon_i a \lambda(g)
  \\
    \lambda_1(g)\mu_ib=\epsilon_i b \lambda(g)

 \end{array}
  \right.
  \]
Since $a\neq 0\neq b$ we get in particular that
$\lambda_1(g)=\lambda_2(g)$ for each $g\in G.$  Recall that $G\subset SL(V),$hence
$\lambda_1(g)\cdot \lambda_2(g) \cdot \lambda(g)=1.$ Let
$m=<x,aw_1-bw_2,v>$ be the unique graded maximal ideal of
$S(V)^{T(G)}=F[x,aw_1-bw_2,v].$ Then, \[g|m/m^2=\left(
\begin{array}{ccc}
\lambda_2(g)^{p^n}& 0 & 0 \\
0 & \lambda_1(g)& a\delta_{23}(g)-b\delta_{13}(g) \\
0 & 0 & \lambda(g) \end{array} \right).\] So by [Theorem \ref{T2}]
$S(V)^G$ is Gorenstein iff $\lambda_2(g)^{p^n}\cdot \lambda_1(g) \cdot
\lambda(g)=1.$ Since $\lambda_2(g)=\lambda_1(g)$ and $\lambda_1(g)\cdot
\lambda_2(g) \cdot \lambda(g)=1,$ then $S(V)^G$ is Gorenstein if
and only if $\lambda_1(g)^{p^n-1}=1.$ Since
$\lambda_1(g)=\lambda_2(g),|T(G)|=p^{n},$ then
$|gT(G)|=|g|m/m^2|=|\lambda_1(g)|,$ for $g\in G.$ So $S(V)^G$ is Gorenstein
if and only if  every $gT(G)\in G/T(G)$ has order dividing
$|T(G)|-1=p^n-1.$

\end{proof}

\begin{remark}
The difference between the previous result  and the results of
\cite{Braun3}, is that in \cite{Braun3}, $G/T(G)$ was a cyclic group and here it is
not necessarily so.

\end{remark}

\newpage

\subsection*{Case 2:$dim (span_F X)=2$}
We also assume that
$T(G)=<\sigma_1,...,\sigma_{n_1},\tau_1,...,\tau_{n_2}>,$ where
\[ \sigma_i= \left(
\begin{array}{ccc}
1& 0 & a\mu_i \\
0 & 1& b\mu_i  \\
0 & 0 & 1\end{array} \right)\] and \[ \tau_i= \left(
\begin{array}{ccc}
1& 0 & c\alpha_i  \\
0 & 1& d\alpha_i  \\
0 & 0 & 1\end{array} \right),\] where $\bigl(\begin{smallmatrix} a
\\ b
\end{smallmatrix}\bigr),\bigl(\begin{smallmatrix} c
\\ d
\end{smallmatrix}\bigr)\in\bigl(\begin{smallmatrix} F
\\ F
\end{smallmatrix}\bigr)$ are $F$-linearly independent.Also
$\{\mu_1,...,\mu_{n_1}\}\subset F$ are $\mathbb{F}_p$-linearly
independent and $\{\alpha_1,...,\alpha_{n_2}\}\subset F$ are
$\mathbb{F}_p$-linearly independent.
 Let $T_1(G)=<\sigma_1,...,\sigma_{n_1}>$ and $T_2(G)=<\tau_1,...,\tau_{n_2}>.$ 

$\\$

We shall prove the following result.
\begin{theorem}\label{CT1}
Let $F$ be a field with $charF=p>0.$ Let $G\subset SL(3,F)$ be a
finite subgroup and $V=F^3,$the natural $SL(3,F)-$module with the
above notation, assumption and $\lambda_1(g)=\lambda_2(g)$ $\forall g\in G.$
Then $S(V)^G$ is Gorenstein in the following cases:

\begin{enumerate}[label=(\roman*)]
\item $p=3.$

\item  $|G|$ is prime to $3.$

 \item $F=\mathbb{F}_q,$ where $q=p^s$ and g.c.d$(3,s)=1.$

\end{enumerate}

\end{theorem}

\begin{lemma}\label{CL2}

 $S(V)^{T(G)}$ is a polynomial ring.

\end{lemma}

\begin{proof}
$\sigma_i(w_1)=w_1+b\mu_iv$ and $\sigma_i(w_2)=w_2+a\mu_iv ,$
implying that $:\sigma_i(aw_1-bw_2)=aw_1-bw_2 ,$ for
$i=1,...,n_1.$ Similarly, $\tau_i(cw_1-dw_2)=cw_1-dw_2 ,$ for
$i=1,...,n_2.$ Since $\{aw_1-bw_2, cw_1-dw_2\}$ are linearly
independent (otherwise $\bigl(\begin{smallmatrix} a
\\b
\end{smallmatrix}\bigr)=\lambda\bigl(\begin{smallmatrix} c
\\ d
\end{smallmatrix}\bigr),$ in contradiction to their  $F$-linear independence),
$\{aw_1-bw_2,cw_1-dw_2,v\}$ is a basis of $V$ and they are
$F-$algebraically independent.Now
$\tau_i(aw_1-bw_2)=a\tau_i(w_1)-b\tau_i(w_2)=a(w_1+d\alpha_iv)-b(w_2+c\alpha_iv)=(aw_1-bw_2)+(ad-bc)\alpha_iv.$
Hence by \cite[Lemma 3.13]{Braun3} $S(F(aw_1-bw_2)+Fv)^{T(G)}=F[x,v]$ is a
polynomial ring, where $x$ is homogenous polynomial (in
$(aw_1-bw_2),v $) of degree $p^{n_2},$ having
$(aw_1-bw_2)^{p^{n_2}}$ as one of its monomials.Similarly
$\sigma_i(cw_1-dw_2)=(cw_1-dw_2)+(cb-da)\mu_iv ,$ and again
$S(F(cw_1-dw_2)+Fv)^{T(G)}=F[y,v] ,$ is a polynomial ring, where
$y$ is homogenous polynomial (in $(cw_1-dw_2),v $) of degree
$p^{n_1},$ having $(cw_1-dw_2)^{p^{n_1}}$ as one of its monomials.
Since $\{aw_1-bw_2,cw_1-dw_2,v\}$ are algebraic independent ,  then
$\{x,y,v\}$ are algebraically independent.Since
$(degx)\cdot(degy)\cdot(degv)=p^{n_2}\cdot p^{n_1}\cdot1=|T(G)| ,$
we have by \cite[Theorem 3.9.4]{Kemper1} that $S(V)^{T(G)}=F[x,y,v]$ is a
polynomial ring.

\end{proof}

\newpage
In the next three results we examine the normality of $T_1(G),T_2(G).$ This is not needed for \textbf{Proposition \ref{CP3}} ,\textbf{Proposition  \ref{CL8}}, and \textbf{Corollary \ref{CC2}}.
\begin{lemma}\label{CL3}
Let $G\subset GL(V)$ be a finite subgroup,$T(G)=T_1(G) \times
T_2(G)$ and suppose that $T_k(G)\lhd G $ for $k=1,2.$ Then one
of the following holds:
\begin{enumerate}[label=(\alph*)]
\item $\lambda_2(g)=\lambda_1(g)$ $\forall g\in G.$ \item
$\lambda_2(g)=\delta(g)\lambda_1(g) , $ where $\delta(g)\in
\mathbb{F}_p$ $\forall g\in G.$
\end{enumerate}

\end{lemma}

\begin{proof}

Since $g$ normalizes $T_1(G)$ $\forall g\in G ,$  we have for each
$i=1,...,n_1 :$

\[\gamma_i=g\sigma_ig^{-1}= \left(
\begin{array}{ccc}
1& 0 & \sum_{j=1}^{n_1
}a\mu_jr_{ji} \\ 
0 & 1& \sum_{j=1}^{n_1
}b\mu_jr_{ji}

\\
0 & 0 & 1\end{array} \right),\] where $r_{ji}\in \mathbb{F}_p.$
Since $g\sigma_i=\gamma_ig $ we get  by equating the $(1,3)$ and $(2,3)$ entries
in both sides: \\

\begin{enumerate}
\item[{(1)}] \[\left\{
\begin{array}{ll}
  \lambda_2(g)a\mu_i=\lambda(g)(\sum_{j=1}^{n_1}\mu_jr_{ji})a \\
    \lambda_1(g)b\mu_i=\lambda(g)(\sum_{j=1}^{n_1
    }\mu_jr_{ji})b
 \end{array}
  \right.
  \]
  \\

  Similarly,since $g$ normalizes $T_2(G)$ $\forall g\in G ,$  we have for each
$i=1,...,n_2 :$ \\

\item[{(2)}] \[\left\{
\begin{array}{ll}
  \lambda_2(g)c\alpha_i=\lambda(g)(\sum_{j=1}^{n_2}\alpha_jr_{ji})c \\
    \lambda_1(g)d\alpha_i=\lambda(g)(\sum_{j=1}^{n_2
    }\alpha_jr_{ji})d

 \end{array}
  \right.
  \]

\end{enumerate}
Recall that  $\bigl(\begin{smallmatrix} a
\\ b
\end{smallmatrix}\bigr),\bigl(\begin{smallmatrix} c
\\ d
\end{smallmatrix}\bigr)\in\bigl(\begin{smallmatrix} F
\\ F
\end{smallmatrix}\bigr)$ are $F$-linearly independent. If $a\neq 0\neq b$ (or $c\neq 0\neq d$ ) then we get
by $\textbf{(1)}$(or $\textbf{(2)}$) that :
$\lambda_1(g)=\lambda_2(g)$ $\forall g\in G.$ On the other hand,if
$\bigl(\begin{smallmatrix} a&c \\ b&d
\end{smallmatrix} \bigr)= \bigl(\begin{smallmatrix} a&0 \\ 0&d
\end{smallmatrix} \bigr),$
 then from $\textbf{(1)}$ and
$\textbf{(2)}$ we get that
$\delta(g)=\frac{\lambda_2(g)}{\lambda_1(g)}=\frac{\alpha_i
\sum_{j=1}^{n_1}\mu_jr_{ji}}{\mu_i\sum_{j=1}^{n_2 }\alpha_j
r_{ji}},$ and if $\bigl(\begin{smallmatrix} a&c \\ b&d
\end{smallmatrix} \bigr)= \bigl(\begin{smallmatrix} 0&c \\ b&0
\end{smallmatrix} \bigr),$ then from $\textbf{(1)}$ and
$\textbf{(2)}$ we get that
$\delta(g)=\frac{\lambda_2(g)}{\lambda_1(g)}=\frac{\mu_i\sum_{j=1}^{n_2 }\alpha_j
r_{ji}}{\alpha_i
\sum_{j=1}^{n_1}\mu_jr_{ji}}.$ Since $\mu_i,\alpha_i,r_{ji}\in \mathbb{F}_p $ then
$\delta(g)\in \mathbb{F}_p .$

\end{proof}

\begin{lemma}\label{CL4}
$T_k(G)\lhd G $ for $k=1,2 $ in the two following cases:
\begin{enumerate}[label=(\alph*)]
\item $\lambda_2(g)=\lambda_1(g)$ $\forall g\in G.$ \item
$\bigl(\begin{smallmatrix} a&c \\ b&d
\end{smallmatrix} \bigr)\in \{\bigl(\begin{smallmatrix} 0&c \\ b&0
\end{smallmatrix} \bigr),\bigl(\begin{smallmatrix} a&0 \\ 0&d
\end{smallmatrix} \bigr)\}.$

\end{enumerate}

\end{lemma}

\begin{proof}
Suppose to the contrary that $T_1(G) \ntriangleleft G.$ Then
$gT(G)g^{-1}=T(G)$ implies that
$g\sigma_ig^{-1}=\sigma_1^{e_1}\cdot...\cdot
\sigma_{n_1}^{e_{n_1}}\cdot \tau_1^{f_1}\cdot...\cdot \tau_{n_2}^
{f_{n_2}} ,$ where not all of the $f_i's$ are $0.$
 By a simple calculation we have:
 \[g\sigma_ig^{-1}=\left(
\begin{array}{ccc}
1& 0 & \lambda(g)^{-1}\mu_i\lambda_2(g)a\\
0 & 1& \lambda(g)^{-1}\mu_i\lambda_1(g)b \\
0 & 0 & 1\end{array} \right).\]

On the other hand we have :

         \[ \sigma_1^{e_1}\cdot...\cdot
\sigma_{n_1}^{e_{n_1}}= \left( \begin{array}{ccc} 1& 0 &
$$\sum_{i=1}^{n_1
}a\mu_ie_i$$ \\
0 & 1& $$\sum_{i=1}^{n_1
}b\mu_ie_i$$\\
0 & 0 & 1\end{array} \right).\]
\\ as well as \\
        \[\tau_1^{f_1}\cdot...\cdot \tau_{n_2}^
{f_{n_2}} = \left( \begin{array}{ccc} 1& 0 & $$\sum_{i=1}^{n_2
}c\alpha_if_i$$ \\
0 & 1& $$\sum_{i=1}^{n_2
}d\alpha_if_i$$ \\
0 & 0 & 1 \end{array} \right).\] \\

Hence by equating the $(1,3)$ and $(2,3)$ entries
in both sides : \[\left\{
\begin{array}{ll}
  \lambda(g)^{-1}\mu_i\lambda_2(g)a=\sum_{j=1}^{n_1
}a\mu_je_j+\sum_{j=1}^{n_2
}c\alpha_jf_j \\
    \lambda(g)^{-1}\mu_i\lambda_1(g)b=\sum_{j=1}^{n_1
}b\mu_je_j+\sum_{j=1}^{n_2 }d\alpha_jf_j

 \end{array}
  \right.
  \]

We firstly assume that $\lambda_2(g)=\lambda_1(g).$ Then
$\lambda_2(g)\lambda(g)^{-1}\mu_i(\begin{smallmatrix} a
\\ b
\end{smallmatrix}\bigr)=(\sum_{j=1}^{n_1
}\mu_je_j)(\begin{smallmatrix} a
\\ b
\end{smallmatrix}\bigr)+(\sum_{j=1}^{n_2
}\alpha_jf_j)(\begin{smallmatrix} c
\\ d
\end{smallmatrix}\bigr).$ So since $(\begin{smallmatrix} a
\\ b
\end{smallmatrix}\bigr),(\begin{smallmatrix} c
\\ d
\end{smallmatrix}\bigr)$ are $F$-linearly independent we get that
$\sum_{j=1}^{n_2 }\alpha_jf_j=0 .$ Since
$\{\alpha_1,...,\alpha_{n_2}\}$ are $\mathbb{F}_p$-linearly
independent ,
$f_j=0$ for $j=1,...,n_2.$ A contradiction, hence $T_1(G)\lhd G.$  \\ \\
Assume now that item \textbf{(b)} holds.Say
$\bigl(\begin{smallmatrix} a&c
\\ b&d
\end{smallmatrix} \bigr)=\bigl(\begin{smallmatrix} a&0 \\ 0&d
\end{smallmatrix} \bigr),$ hence $(\begin{smallmatrix}
\lambda(g)^{-1}\mu_i\lambda_2(g)a
\\ 0
\end{smallmatrix}\bigr)=(\sum_{j=1}^{n_1
}\mu_je_j)(\begin{smallmatrix} a
\\ 0
\end{smallmatrix}\bigr)+(\sum_{j=1}^{n_2
}\alpha_jf_j)(\begin{smallmatrix} 0
\\ d
\end{smallmatrix}\bigr),$ so $(\sum_{j=1}^{n_2
}\alpha_jf_j)=0 .$  Since
$\{\alpha_1,...,\alpha_{n_2}\}$ are $\mathbb{F}_p$-linearly
independent, $f_j=0$ for $j=1,...,n_2,$ a contradiction. If $\bigl(\begin{smallmatrix} a&c
\\ b&d
\end{smallmatrix} \bigr)=\bigl(\begin{smallmatrix} 0&c \\ b&0
\end{smallmatrix} \bigr),$then by the same steps above we get that
$(\sum_{j=1}^{n_2 }\alpha_jf_j)=0,$ again $f_j=0$ for $j=1,...,n_2
,$ again a contradiction.By a similar discussion we get that
$T_2(G)\lhd G.$

\end{proof}

The combination of \textbf{Lemma \ref{CL3}} and \textbf{Lemma \ref{CL4}}
yields.

\begin{proposition}\label{CP2}
The following conditions are equivalent .

\begin{enumerate}[label=(\alph*)]
\item $T_k(G)\lhd G,$ for $k=1,2.$  \item one of the following
holds:

\end{enumerate}

\begin{enumerate}
\item $\bigl(\begin{smallmatrix} a&c
\\ b&d
\end{smallmatrix} \bigr)\in \{\bigl(\begin{smallmatrix} 0&c
\\ b&0
\end{smallmatrix} \bigr),\bigl(\begin{smallmatrix} a&0
\\ 0&d
\end{smallmatrix} \bigr)\}$
\item $\lambda_2(g)=\lambda_1(g),$ for all $g\in G.$

\end{enumerate}

\end{proposition}

\begin{proof}
That \textbf{(b)} implies \textbf{(a)} is a consequence of \textbf{Lemma \ref{CL4}}. That  \textbf{(a)} implies \textbf{(b)} follows from the proof of  \textbf{Lemma \ref{CL3}}.

\end{proof}

\newpage

We can now give a necessary and sufficient condition for a
subgroup $G$ such that $S(V)^G$ is Gorenstein.

\begin{proposition}\label{CP3}
Suppose $n_1<n_2.$ Then $S(V)^G$ is Gorenstein if and only if :

\begin{enumerate}
\item $\lambda_1(g)^{p^{n_1}+p^{n_2}}\lambda(g)=1,$ in case $c\neq
0\neq d.$  \item
$\lambda_1(g)^{p^{n_2}}\lambda_2(g)^{p^{n_1}}\lambda(g)=1,$ in
case $c=0 \neq d.$  \item
$\lambda_2(g)^{p^{n_2}}\lambda_1(g)^{p^{n_1}}\lambda(g)=1,$ in
case $c\neq 0=d.$ \end{enumerate}

The case $n_2<n_1$ is similar.Note that we do not assume any normality assumption on $T_k(G).$

\end{proposition}

\begin{proof}

 By  \textbf{Lemma \ref{CL2}} , we have that
$S(V)^{T(G)}=F[x,y,v],deg(y)=p^{n_1},deg(x)=p^{n_2}.$ Since
$g^{-1}T(G)g=T(G)$ we have $g(y)\in S(V)^{T(G)}.$ Therefore,since
$g$ preserve degrees,$g(y)=\alpha y+\beta v^{p^{n_1}}.$ Recall that
$(cw_1-dw_2)^{p^{n_1}}$ appears as a monomial in $y$ and all other
monomials involve $v$ in addition.\\\\
We have :\\\\
$g(cw_1-dw_2)=cg(w_1)-dg(w_2)=c(\lambda_1(g)w_1+\delta_{23}(g)v)-d(\lambda_2(g)w_2+\delta_{13}(g)v)=(c\lambda_1(g)w_1-d\lambda_2(g)w_2)+
\chi v,$ where $\chi=c\delta_{23}(g)-d\delta_{13}(g).$ Hence
$g((cw_1-dw_2)^{p^{n_1}})=c^{p^{n_1}}\lambda_1(g)^{p^{n_1}}w_1^{p^{n_1}}-
d^{p^{n_1}}\lambda_2(g)^{p^{n_1}}w_2^{p^{n_1}}+\chi^{p^{n_1}}v^{p^{n_1}}.$ So by comparing coefficients of $y$ in $g(y)$ we get:\\

$c^{p^{n_1}}\lambda_1(g)^{p^{n_1}}w_1^{p^{n_1}}-d^{p^{n_1}}\lambda_2(g)^{p^{n_1}}w_2^{p^{n_1}}=
\alpha (cw_1-dw_2)^{p^{n_1}}.$ Consequently
\begin{enumerate}

\item [\textbf{(1)}] $c^{p^{n_1}}\lambda_1(g)^{p^{n_1}}=\alpha
c^{p^{n_1}} , d^{p^{n_1}}\lambda_2(g)^{p^{n_1}}=\alpha d^{p^{n_1}}. $

\end{enumerate}

 Similarly $g(x)=\delta x+\varepsilon y^{p^{n_2-n_1}}+$ mixed terms (in
$y,v$)+$\xi v^{p^{n_2}}.$ Recall that $(aw_1-bw_2)^{p^{n_2}}$ appears as a monomial in $x$ and all other
monomials involve $v$ in addition.So the contribution to the terms with
$w_1^{p^{n_2}}$ in the right hand side of the last equation are  $\delta
(aw_1-bw_2)^{p^{n_2}}+\varepsilon
((cw_1-dw_2)^{p^{n_1}})^{p^{n_2-n_1}}= \delta
(aw_1-bw_2)^{p^{n_2}}+\varepsilon (cw_1-dw_2)^{p^{n_2}}.\\$
The contribution of the left hand side are from
$g((aw_1-bw_2)^{p^{n_2}})=a^{p^{n_2}}(\lambda_1(g)w_1+
\delta_{23}(g)v)^{p^{n_2}}-b^{p^{n_2}}(\lambda_2(g)w_2+
\delta_{13}(g)v)^{p^{n_2}}.$ Consequently we get :

\begin{enumerate}

\item [\textbf{(2)}]$a^{p^{n_2}}\lambda_1(g)^{p^{n_2}}=\delta
a^{p^{n_2}}+\varepsilon
c^{p^{n_2}},b^{p^{n_2}}\lambda_2(g)^{p^{n_2}}= \delta
b^{p^{n_2}}+\varepsilon d^{p^{n_2}}.$

\end{enumerate}

 If $c\neq 0 \neq d,$ then \textbf{(1)} shows that
$\lambda_1(g)^{p^{n_1}}=\alpha=\lambda_2(g)^{p^{n_1}},$ so
$\lambda_1(g)=\lambda_2(g).$ Now \textbf{(2)} shows that
$a^{p^{n_2}}(\lambda_1(g)^{p^{n_2}}-\delta)=\varepsilon
c^{p^{n_2}},
b^{p^{n_2}}(\lambda_2(g)^{p^{n_2}}-\delta)=
\varepsilon d^{p^{n_2}}.\\$
So if $\varepsilon \neq 0 $ then
$\frac{a^{p^{n_2}}}{\varepsilon c^{p^{n_2}}}=
\frac{b^{p^{n_2}}}{\varepsilon d^{p^{n_2}}}.$ Hence
$\frac{a}{c}=\frac{b}{d}$ or $\frac{a}{b}=\frac{c}{d},$ which is a
contradiction,so $\varepsilon=0 $ in case $c\neq 0\neq d.$  Therefore
$\lambda_1(g)^{p^{n_2}}=\delta=\lambda_2(g)^{p^{n_2}}.$ Let
$m=<x,y,v>$  be the unique homogenous maximal ideal of
$S(V)^{T(G)}=F[x,y,v].$ The matrix representing $g$ on $m/m^2$
whenever $c\neq 0 \neq d$ is
 \[ \left(
\begin{array}{ccc}
\lambda_1(g)^{p^{n_2}}& 0 & 0 \\
0 & \lambda_1(g)^{p^{n_1}}& 0  \\
0 & 0 & \lambda(g)\end{array} \right),\] so the condition for
$S(V)^G$ to be Gorenstein is
$\lambda_1(g)^{p^{n_2}+p^{n_1}}\lambda(g)=1.$ Since $G\subset
SL(V),\lambda_1(g)=\lambda(g)^{-2}.$ Hence $S(V)^G$ is Gorenstein
if and only if $\lambda_1(g)^{p^{n_2}+p^{n_1}-2}=1.$ Moreover, by \textbf{Lemma \ref{CL4}}, $T_k(G)\lhd G $ for $k=1,2.$\\
Assume now that $c=0\neq d,$
hence from \textbf{(1)} we get $\lambda_2(g)^{p^{n_1}}=\alpha$ and from
\textbf{(2)}, since $a\neq 0$ in this case, $\lambda_1(g)^{p^{n_2}}=\delta.$ So the matrix representing
$g$ on $m/m^2$ is
\[ \left(
\begin{array}{ccc}
\lambda_1(g)^{p^{n_2}}& 0 & 0 \\
0 & \lambda_2(g)^{p^{n_1}}& 0  \\
0 & 0 & \lambda(g)\end{array} \right),\] and $S(V)^G$ is
Gorenstein if and only if
$\lambda_1(g)^{p^{n_2}}\lambda_2(g)^{p^{n_1}}\lambda(g)=1.$

Assume now that $c\neq 0=d.$ Then from \textbf{(1)} we get
$\lambda_1(g)^{p^{n_1}}=\alpha$ and from
\textbf{(2)}, since $b\neq 0$ in this case, $\lambda_2(g)^{p^{n_2}}=\delta.$ So the matrix representing
$g$ on $m/m^2$ is
\[ \left(
\begin{array}{ccc}
\lambda_2(g)^{p^{n_2}}& 0 & 0 \\
0 & \lambda_1(g)^{p^{n_1}}& 0  \\
0 & 0 & \lambda(g)\end{array} \right),\] and $S(V)^G$ is
Gorenstein if and only if
$\lambda_2(g)^{p^{n_2}}\lambda_1(g)^{p^{n_1}}\lambda(g)=1.$ This
settles the case $n_1<n_2.$ If $n_1>n_2$ then we do the same
analysis but with $\{a,b\}$ instead of $\{c,d\}.$

\end{proof}

Our next result will be used in the proof of \textbf{Theorem \ref{CT1}}.

\begin{lemma}\label{CL6}
Suppose $T_k(G)\lhd G $.Then
$(\frac{\lambda_2(g)}{\lambda(g)})^{p^{n_k}-1}=(\frac{\lambda_1(g)}{\lambda(g)})^{p^{n_k}-1}=1
,$ for $k=1,2 .$

\end{lemma}

\begin{proof}
We have by \textbf{Lemma \ref{CL2}}  that
$S(F(aw_1-bw_2)+Fv)^{T_2(G)}=F[x,v] ,$  where $x$ is homogenous
polynomial (in $(aw_1-bw_2),v $) of degree $p^{n_2},$ having
$(aw_1-bw_2)^{p^{n_2}}$ as one of its monomials and all other monomials involve $v.$
In other words $x=(aw_1-bw_2)^{p^{n_2}}+v^{p^{n_2}-1}(aw_1-bw_2)+$ monomials involve $v.$
By a simple calculating we get that :\\

\[\left\{
\begin{array}{ll}
 g(aw_1-bw_2)^{p^{n_2}}=a^{p^{n_2}}\lambda_1(g)^{p^{n_2}}w_1^{p^{n_2}}-b^{p^{n_2}}\lambda_2(g)^{p^{n_2}}w_2^{p^{{n_2}}}+(a\delta_{23}(g)-b\delta_{13}(g))^{p^{n_2}}v^{p^{n_2}}.\\
    g(v^{p^{n_2}-1}(aw_1-bw_2))=a\lambda(g)^{p^{n_2}-1}\lambda_1(g)w_1v^{p^{n_2}-1}-
    b\lambda(g)^{p^{n_2}-1}\lambda_2(g)w_2v^{p^{n_2}-1}+\lambda(g)^{p^{n_2}-1}v^{p^{n_2}}(a\delta_{23}(g)-b\delta_{13}(g)).

 \end{array}
  \right.
  \]

Now $T_2(G)\lhd G$ imply that $F[x,v]$ is
$g-$stable.Therfeore , since $g$ preserve degrees , we get $g(x)=\rho
x+\kappa v^{p^{n_2}} ,\rho,\kappa\in F.$ Assume firstly,using
\textbf{Lemma \ref{CL3}}, that $\lambda_1(g)=\lambda_2(g).$
Hence , if $a\neq 0$ then by comparing coefficients of $w_1^{p^{n_2}}$ and $v^{p^{n_2}-1}w_1$ in $g(x)$ we get
$\rho=\lambda_1(g)^{p^{n_2}}=\lambda(g)^{p^{n_2}-1}\lambda_1(g).$
So $(\frac{\lambda_1(g)}{\lambda(g)})^{p^{n_2}-1}=1,$ and if $b\neq 0$ then by comparing coefficients of $w_2^{p^{n_2}}$ and $v^{p^{n_2}-1}w_2$ in $g(x)$ we get
$\rho=\lambda_2(g)^{p^{n_2}}=\lambda(g)^{p^{n_2}-1}\lambda_2(g).$ So $(\frac{\lambda_2(g)}{\lambda(g)})^{p^{n_2}-1}=1 .\\$

Assume now, using \textbf{Lemma \ref{CL3}}, that 
$\lambda_2(g)=\delta(g)\lambda_1(g),$ where $\delta(g)\in
\mathbb{F}_p.$ Consequently, by \textbf{Proposition \ref{CP2}} , $\bigl(\begin{smallmatrix} a&c \\ b&d
\end{smallmatrix} \bigr)\in \{\bigl(\begin{smallmatrix} 0&c \\ b&0
\end{smallmatrix} \bigr),\bigl(\begin{smallmatrix} a&0 \\ 0&d
\end{smallmatrix} \bigr)\}.$ Hence if $a\neq 0=b ,$ then by  comparing coefficient of $w_1^{p^{n_2}}$ and $v^{p^{n_2}-1}w_1$ in $g(x)$ we get
$\rho=\lambda_1(g)^{p^{n_2}}=\lambda(g)^{p^{n_2}-1}\lambda_1(g).$ So $(\frac{\lambda_1(g)}{\lambda(g)})^{p^{n_2}-1}=1 ,$ hence $(\frac{\delta(g)^{-1}\lambda_2(g)}{\lambda(g)})^{p^{n_2}-1}=1.$  By
little Fermat's theorem we get that $\delta(g)^{p^{n_2}-1}=1 ,$ then
$(\frac{\lambda_2(g)}{\lambda(g)})^{p^{n_2}-1}=1.$
If $b\neq 0=a ,$ then by  comparing coefficient of $w_2^{p^{n_2}}$ and $v^{p^{n_2}-1}w_2$ in $g(x)$ we get $\rho=\lambda_2(g)^{p^{n_2}}=\lambda(g)^{p^{n_2}-1}\lambda_2(g).$ So $(\frac{\lambda_2(g)}{\lambda(g)})^{p^{n_2}-1}=1 ,$ hence $(\frac{\delta(g)\lambda_1(g)}{\lambda(g)})^{p^{n_2}-1}=1.$  By
little Fermat's theorem we get that $\delta(g)^{p^{n_2}-1}=1 ,$ then
$(\frac{\lambda_2(g)}{\lambda(g)})^{p^{n_2}-1}=1 .\\$

 Similarly, if we look at $S(F(cw_1-dw_2)+Fv)^{T_1(G)}=F[y,v]$ from
\textbf{Lemma \ref{CL2}} , then the above reasoning shows again that
$(\frac{\lambda_k(g)}{\lambda(g)})^{p^{n_1}-1}=1 ,$ for $k=1,2.$

\end{proof}

\begin{lemma}\label{CL7}
 $(|G|/|T(G)|,p)=1,$ and if $\lambda_1(g)=\lambda_2(g)$ $\forall
g\in G$ then $G/T(G)$ is a cyclic group.

\end{lemma}

\begin{proof}
Since $G\subset
SL(V),\lambda(g)=\lambda_1(g)^{-1}\lambda_2(g)^{-1}.$
 We have:
\[g^p= \left(
\begin{array}{ccc}
\lambda_2(g)^p&0& * \\
0 & \lambda_1(g)^p& *  \\
0 & 0 & \lambda_1(g)^{-p}\lambda_2(g)^{-p}\end{array} \right)\]

If $g^p\in T(G)$ then $\lambda_1(g)^p=\lambda_2(g)^p=1 ,$ and then
$\lambda_1(g)=\lambda_2(g)=1 ,$ so :
\[g=\left(
\begin{array}{ccc}
1&0 & * \\
0 & 1& *  \\
0 & 0 & 1\end{array} \right)\in T(G).\] This shows that
$(|G|/|T(G)|,p)=1.$ Assume now that $\lambda_1(g)=\lambda_2(g)$
$\forall g\in G.$ Set
$U=\mathbb{F}_p\mu_1+\cdot\cdot\cdot\mathbb{F}_p\mu_{n_1}.$ By
\textbf{Lemma \ref{CL4}}, $T_1(G)\lhd G$ then $g\sigma_ig^{-1}\in
T_1(G).$
 Now since \[g\sigma_ig^{-1}=\left(
\begin{array}{ccc}
1& 0 & \lambda(g)^{-1}\mu_i\lambda_2(g)a\\
0 & 1& \lambda(g)^{-1}\mu_i\lambda_1(g)b \\
0 & 0 & 1\end{array} \right)\] for $i=1,...,n_1,$ we get that
$\lambda(g)^{-1}\lambda_1(g)\mu_i\in U,$ hence $\lambda(g)^{-1}\lambda_1(g)U\subseteq U$  Since $G\subset SL(V)$ and
$\lambda_1(g)=\lambda_2(g)$ then
$\lambda(g)^{-1}\lambda_1(g)=\lambda_1(g)^{3}.$Consequently by
\cite[Lemma 3.12.] {Braun3} $|\lambda_1(g)^3|$ divides $(p^{n_1}-1).$ By a
similar discussion on $T_2(G)$ we get that $|\lambda_1(g)^3|$
divides $(p^{n_2}-1).$ Let $H=\{\epsilon\in F|\epsilon^{3(p^{n_1}-1)}=1\}.$ Define $\Phi:G/T(G)\longrightarrow H,$
by $\Phi(gT(G))=\lambda_1(g),$ for each $g\in G.$ Since \[gT(G)=g|m/m^2=\left(
\begin{array}{ccc}
\lambda_1(g)^{p^{n_2}}& 0 & 0\\
0 & \lambda_1(g)^{p^{n_1}}& 0 \\
0 & 0 & \lambda_1(g)^{-2}\end{array} \right),\]
it follows that $\Phi$ is an injective homomorphism with its image in $H$ since  $|gT(G)|=|\lambda_1(g)|$ divides $3(p^{n_1}-1).$  Since $H$ is cyclic group it follows that $G/T(G)$ is cyclic.

\end{proof}
\newpage
 Recall that $n_1$ and $n_2$ are the numbers of the generators of
$T_1(G)$ and $T_2(G)$ respectively.\\
The next result considers the case $n_1=n_2$ which was omitted
from \textbf{Proposition \ref{CP3}}.

\begin{proposition}\label{CL8}

Suppose that $n_1=n_2.$ Then $S(V)^G$ is Gorenstein iff
$\lambda_2(g)^{p^{n_1}}\lambda_1(g)^{p^{n_1}}\lambda(g)=1.$

\end{proposition}

\begin{proof}
We have by \textbf{Lemma \ref{CL2}} that
$S(V)^{T(G)}=F[x,y,v],deg(y)=p^{n_1},deg(x)=p^{n_2}.$ Since
$g^{-1}T(G)g=T(G)$ we have $g(y),g(x)\in S(V)^{T(G)}.$ Hence,
since $g$ preserve degrees, we get :\\

\begin{enumerate}
\renewcommand{\labelenumi}{(\arabic{enumi})~}
\item $g(y)=\alpha y+\gamma x+\beta v^{p^{n_1}} , \alpha,\gamma,\beta\in F .$
\item $g(x)=\delta x+\varepsilon y+ \chi v^{p^{n_1}} , \delta,\varepsilon,\chi \in F .$
\end{enumerate}

 Therefore, $S(V)^G=[S(V)^{T(G)}]^{G/T(G)}.$ Let $m=(x,y,v)$ be the unique graded maximal
ideal of $S(V)^{T(G)}=F[y,x,v] ,$ then the matrix of $\bar{g}\in
G/T(G)$ on $m/m^2=F\bar{y}+F\bar{x}+F\bar{v}$ is :
\[\left(
\begin{array}{ccc}
\alpha& \gamma & 0\\
\varepsilon& \delta& 0\\
0 & 0 & \lambda(g)\end{array} \right).\] So by [Theorem \ref{T2}]
$S(V)^G$ is Gorenstein if and only if $(\alpha \delta-\varepsilon
\gamma)\lambda(g)=1$
$\forall g\in G.$\\\\
\textbf{The computation of} $\alpha \delta-\varepsilon \gamma :$\\\\
We shall assume that $c\neq 0\neq d$ and $a\neq 0\neq b.$ We know
that $(cw_1-dw_2)^{p^{n_1}}$ appears as a monomial in $y,$ 
$(aw_1-bw_2)^{p^{n_1}}$ appears as a monomial in $x,$ and all
other monomials involve also $v.$ Therefore the coefficients of
$w_1^{p^{n_1}}$ in $\textbf{(1)}$ are :\\
$c^{p^{n_1}}\lambda_1(g)^{p^{n_1}}$ in the left hand side and
$\alpha c^{p^{n_1}}+\gamma a^{p^{n_1}}$ in the right hand side.Hence
$c^{p^{n_1}}\lambda_1(g)^{p^{n_1}}=\alpha c^{p^{n_1}}+\gamma
a^{p^{n_1}}.$ Similarly, for the coefficients of $w_2^{p^{n_1}},$
we have $d^{p^{n_1}}\lambda_2(g)^{p^{n_1}}=\alpha
d^{p^{n_1}}+\gamma b^{p^{n_1}}$.From $\textbf{(2)}$ we similarly
have:\\

\[\left\{
\begin{array}{ll}
 a^{p^{n_1}}\lambda_1(g)^{p^{n_1}}=\delta a^{p^{n_1}}+\varepsilon  c^{p^{n_1}}  \\
  b^{p^{n_1}}\lambda_2(g)^{p^{n_1}}=\delta b^{p^{n_1}}+\varepsilon  d^{p^{n_1}}

 \end{array}
  \right.
  \]

$$\\$$Hence from $\textbf{(1)}$ and $\textbf{(2)}$  we have :\\

\[\left\{
\begin{array}{llll}
 \lambda_1(g)^{p^{n_1}}=\alpha +\gamma (\frac{a}{c})^{p^{n_1}} \\
 \lambda_2(g)^{p^{n_1}}=\alpha +\gamma (\frac{b}{d})^{p^{n_1}} \\
\lambda_1(g)^{p^{n_1}}=\delta+\varepsilon (\frac{c}{a})^{p^{n_1}}
\\
\lambda_2(g)^{p^{n_1}}=\delta+\varepsilon (\frac{d}{b})^{p^{n_1}}

 \end{array}
  \right.
  \]

$\\$ Hence,
\[\left\{
\begin{array}{ll}
 (\lambda_1(g)-\lambda_2(g))^{p^{n_1}}=\gamma
(\frac{a}{c}-\frac{b}{d})^{p^{n_1}} \\

 (\lambda_1(g)-\lambda_2(g))^{p^{n_1}}=\varepsilon (\frac{c}{a}-\frac{d}{b})^{p^{n_1}}
\end{array}
  \right.
  \]
  \newpage
  
It follows that:
\begin{enumerate}[label=(\Alph*)]
\item := $(\lambda_1(g)-\lambda_2(g))^{2p^{n_1}}=\gamma
\varepsilon((\frac{a}{c}-\frac{b}{d})(\frac{c}{a}-\frac{d}{b}))^{p^{n_1}}.$

$\\$ Also we have :

\[\left\{
\begin{array}{ll}
 (\frac{c}{a}\lambda_1(g)-\frac{d}{b}\lambda_2(g))^{p^{n_1}}=\alpha
(\frac{c}{a}-\frac{d}{b})^{p^{n_1}} \\

 (\frac{a}{c}\lambda_1(g)-\frac{b}{d}\lambda_2(g))^{p^{n_1}}=\delta (\frac{a}{c}-\frac{b}{d})^{p^{n_1}}
\end{array}
  \right.
  \]
And it follows that : \item:=
$(\frac{c}{a}\lambda_1(g)-\frac{d}{b}\lambda_2(g))^{p^{n_1}}(\frac{a}{c}\lambda_1(g)-\frac{b}{d}\lambda_2(g))^{p^{n_1}}=
\alpha
\delta[(\frac{c}{a}-\frac{d}{b})(\frac{a}{c}-\frac{b}{d})]^{p^{n_1}}.$\\
\end{enumerate}
So $\textbf{(B)-(A)}=(\alpha \delta-\varepsilon
\gamma)[(\frac{c}{a}-\frac{d}{b})(\frac{a}{c}-\frac{b}{d})]^{p^{n_1}}.$
Observe that
$(\frac{c}{a}-\frac{d}{b})(\frac{a}{c}-\frac{b}{d})=2-(\frac{cb}{ad}+\frac{ad}{bc}).$
Also
$\textbf{(B)-(A)}=(\lambda_1(g)^2+\lambda_2(g)^2-(\frac{cb}{ad}+\frac{ad}{bc})\lambda_1(g)\lambda_2(g))^{p^{n_1}}-
(\lambda_1(g)^2+\lambda_2(g)^2-2\lambda_1(g)\lambda_2(g))^{p^{n_1}}=(\lambda_1(g))^{p^{n_1}}(\lambda_2(g))^{p^{n_1}}
[2-(\frac{cb}{ad}+\frac{ad}{bc})]^{p^{n_1}}.$

$\\$ Hence $\frac{(B)-(A)}{((\frac{c}{a}-\frac{d}{b})(\frac{a}{c}-\frac{b}{a}))^{p^{n_1}}}=(\alpha
\delta-\varepsilon \gamma).$ Therefore $\alpha \delta-\varepsilon
\gamma=\lambda_1(g)^{p^{n_1}}\lambda_2(g)^{p^{n_1}},$ hence
$S(V)^G$ is Gorenstein if and only if
$\lambda_2(g)^{p^{n_1}}\lambda_1(g)^{p^{n_1}}\lambda(g)=1.\\$

 The cases where one of $\{a,b,c,d\}$ is $0$ are easier.For
example, if $a=0,$ hence since $\bigl(\begin{smallmatrix} a
\\ b
\end{smallmatrix}\bigr),\bigl(\begin{smallmatrix} c
\\ d
\end{smallmatrix}\bigr)$ are $F$-linearly independent , $c\neq 0 ,$ .Hence
$\lambda_1(g)^{p^{n_1}}=\alpha$ and $\varepsilon=0,$ therefore
$\delta=\lambda_2(g)^{p^{n_1}}.$ Hence $\alpha \delta-\varepsilon
\gamma=\alpha
\delta=\lambda_1(g)^{p^{n_1}}\lambda_2(g)^{p^{n_1}}.$So again
$S(V)^G$ Gorenstein if and only if
$\lambda_2(g)^{p^{n_1}}\lambda_1(g)^{p^{n_1}}\lambda(g)=1.$
Equivalently $\lambda(g)^{p^{n_1}-1}=1.$

\end{proof}

We finally arrive at: \underline{\textbf{The proof of \textbf{Theorem \ref{CT1}.}}}\\

\begin{proof}

We have by \textbf{Lemma  \ref{CL7}} that $(|G|/|T(G)|,p)=1$ and
$G/T(G)\cong <\lambda_1(g)>$ is a cyclic group.
 Let $\frac{|G|}{|T(G)|}=|\lambda_1(g)|\equiv k ,$ since $G\subset
SL(V),\frac{\lambda_1(g)}{\lambda(g)}=\frac{\lambda_1(g)}{\lambda_1(g)^{-1}\lambda_2(g)^{-1}}=\frac{\lambda_1(g)}{\lambda_1(g)^{-1}\lambda_1(g)^{-1}}=\lambda_1(g)^3.$
Hence,by
\textbf{Lemma \ref{CL6}} , $(\lambda_1(g)^3)^{p^{n_1}-1}=(\lambda_1(g)^3)^{p^{n_2}-1}=1.$So
$(\lambda_1(g)^3)^{p^{n_1}+p^{n_2}-2}=1,$ which means that  $k$
divides  $[3\cdot({p^{n_1}+p^{n_2}-2})].$\\

If $p=3$ then  since $(3,k)=1, k$ divides $(p^{n_1}+p^{n_2}-2),$
consequently by \textbf{Proposition \ref{CP3}} and \textbf{Proposition \ref{CL8}},  $S(V)^G$ is
Gorenstein.This settles item $(i).$\\

Assume now that $3\nmid |G|.$The action of $g^3$ on $m/m^2$ where
$m$ is the unique maximal ideal of $S(V)^{T(G)}$  is : \[\left(
\begin{array}{ccc}
\lambda_1(g)^{3\cdot {p^{n_2}}}& 0& 0\\
0& \lambda_1(g)^{3\cdot {p^{n_1}}}& 0\\
0 & 0 & \lambda_1(g)^{-6}\end{array} \right).\]

Since  $|\lambda_1(g)^3|$ divides $p^{n_1}+p^{n_2}-2 ,$ the above reasoning shows that $S(V)^{<g^3,T(G)>}=S(m/m^2)^{<g^3|m/m^2>}$ is
Gorenstein.Consequently if $3\nmid |G|,$ then $G=<g^3,T(G)>$ and
$S(V)^G$ is therefore Gorenstein,so item $(ii)$ is verified.\\\\
We have by \textbf{Lemma \ref{CL4}} that  $T_k(G)\lhd G $ for $k=1,2
,$ hence by using the proof of  \textbf{Lemma \ref{CL3}} we get :

\begin{enumerate}
\item[\textbf{(1)}] \[\left\{
\begin{array}{ll}
  \lambda_2(g)a\mu_i=\lambda(g)($$\sum_{j=1}^{n_1}\mu_jr_{ji}$$)a \\
    \lambda_1(g)b\mu_i=\lambda(g)($$\sum_{j=1}^{n_1 }\mu_jr_{ji}$$)b

 \end{array}
  \right.
  \]
  \\

\item[\textbf{(2)}] \[\left\{
\begin{array}{ll}
  \lambda_2(g)c\alpha_i=\lambda(g)($$\sum_{j=1}^{n_2}\alpha_jr_{ji}$$)c \\
    \lambda_1(g)d\alpha_i=\lambda(g)($$\sum_{j=1}^{n_2
    }\alpha_jr_{ji}$$)d

 \end{array}
  \right.
  \]

\end{enumerate}

Given the assumption of item $(iii)$.Since
$\lambda_1(g)=\lambda_2(g),$ it follows that $(a\neq 0\neq b)$ or
$(c\neq 0\neq d).$ Say $a,c,d\neq 0,$ the equations above imply
that $\frac{\lambda(g)}{\lambda_1(g)}($$\sum_{j=1}^{n_1
}\mu_jr_{ji}$$)=\mu_i$ and
$\frac{\lambda(g)}{\lambda_1(g)}($$\sum_{j=1}^{n_2}\alpha_jr_{ji}$$)=\alpha_i.$
Set $M_1=\mathbb{F}_p\mu_1+...+\mathbb{F}_p\mu_{n_1}$ and
$M_2=\mathbb{F}_p\alpha_1+...+\mathbb{F}_p\alpha_{n_2},$ so
$dim_{\mathbb{F}_p}M_1=n_1$ and $dim_{\mathbb{F}_p}M_2=n_2.$
\\
Thus
$\frac{\lambda_1(g)}{\lambda(g)}M_1=\lambda_1(g)^3M_1\subseteq
M_1$ and
$\frac{\lambda_1(g)}{\lambda(g)}M_2=\lambda_1(g)^3M_2\subseteq
M_2.$ \\Suppose to the contrary that $\lambda_1(g)\not\in
\mathbb{F}_p(\lambda_1(g)^3),$ then
$[\mathbb{F}_p(\lambda_1(g)):\mathbb{F}_p(\lambda_1(g)^3)]=3.$
But $[F:\mathbb{F}_p]=s$ is relatively prime to $3 ,$ a
contradiction.Thus $\lambda_1(g)\in \mathbb{F}_p(\lambda_1(g)^3),
$ hence $\lambda_1(g)M_1\subseteq M_1 ,$ similarly,we get that
$\lambda_2(g)M_2\subseteq M_2.$ Therefore, by \cite[Lemma 3.12.]{Braun3}
$k|(p^{n_1}-1)$ and $k|(p^{n_2}-1),$ hence
$k|(p^{n_1}+p^{n_2}-2).$ Consequently,by \textbf{Proposition \ref{CP3}} item $ \textbf{(1)}$ ,
$S(V)^G$ is Gorenstein and item $(iii)$
holds.\\

\end{proof}

We also have here the following.

\begin{corollary}\label{CC2}

Suppose $G\subset SL(3,\mathbb{F}_p)$ and $V=\mathbb{F}_p^3$ contains two $2-$dimensional $G-$submodules. Then $S(V)^G$ is Gorenstein.

\end{corollary}

\begin{proof}
Since $F=\mathbb{F}_p, dim_{{F}_p}(span X)\in \{1,2\},$ and either $T(G)=T_{1}(G)\times T_2(G)$ or $T(G)=T_1(G).$ If $dim_{\mathbb{F}_p}(span X)=2,$ then $n_1=n_2=1.$ So by \textbf{Proposition \ref{CL8}} the condition for being Gorenstein is 
$\lambda_2(g)^p\lambda_1(g)^p\lambda(g)=1,$ for each $g\in G.$ Since $\lambda(g)=\lambda_1(g)^{-1}\lambda_2(g)^{-1},$ it follows that it is equivalent to $\lambda_2(g)^{p-1}\lambda_1(g)^{p-1}=1.$ Now since
$\lambda_1(g),\lambda_2(g)\in \mathbb{F}_p\setminus \{0\}$ it follows that $\lambda_1(g)^{p-1}=1=\lambda_2(g)^{p-1},$ and the condition is fullfilled.
If $dim_{\mathbb{F}_p}(span X)=1,$ then $n=1,$ that is $|T(G)|=p,$ so since $|gT(G)|=|\lambda_1(g)|,$
it follows from $\lambda_1(g)\in \mathbb{F}_p\setminus \{0\}$ that $|gT(G)|$ divides $|T(G)|-1,$
and again the Gorenstein property holds.

\end{proof}

\newpage

\section{Chapter(C): Preliminaries to Chapters D,E,F and G.}
Let $G\subset GL(V)$ be a group and $W\subset V$ a $G-$submodule with $dim_FW=dim_FV-1.$
With respesct to a fixed basis, each $g\in G$ is represented by the matrix :
$$  g= \left(\begin{matrix}
g_{11} &  g_{12}  & \ldots &  g_{1n}\\
0  &   g_{22}& \ldots &  g_{2n}\\
\vdots & \vdots & \ddots & \vdots\\
0  &    g_{n2} & \ldots     & g_{nn}
\end{matrix}\right)$$

\begin{lemma}\label{FL1}
Suppose $g\in G\cap SL(V).$ Then
 $$g^{-1} hg=
\left(\begin{matrix}
1 &   \indent \indent \indent  G_{11}(\alpha_1,...,\alpha_{n-1})g|W \\
0  &  1 & 0 &\ldots &\ldots &\ldots&&0\\
0 &  0 &1  & 0&\ldots&\ldots&0&0&\\
\vdots& \vdots& 0& \vdots&&&\vdots &\vdots\\
0 &\ldots&\vdots&0&\ldots&0 &1  & 0\\
0  &   0  & 0& \ldots &\ldots&\ldots& 0  &1
\end{matrix}\right)=\\$$

$$=\left(\begin{matrix}
1 &  & G_{11}(\alpha_1,...,\alpha_{n-1})g|W  & \\
0  &  \\
\vdots & \indent \indent& I_{n-1} \\
0  &
\end{matrix}\right) ,where \\$$

 $$   \indent  g|W=\left(\begin{matrix}
g_{22} &  g_{23}  & \ldots &  g_{2n}\\
 g_{32}& \ldots &  &  g_{3n}\\
\vdots & \vdots & \ddots & \vdots\\
   g_{n2} & \ldots  &    & g_{nn}
\end{matrix}\right), G_{11}=det(g|W),and$$ 
\begin{center}

$$h=\left(\begin{matrix}
1 &  \alpha_1  & \ldots& \ldots & &&\ldots&\alpha_{n-1}\\
0  &  1 & 0 &\ldots &\ldots &\ldots&&0\\
0 &  0 &1  & 0&\ldots&\ldots&0&0&\\
\vdots& \vdots& 0& &&&\vdots&\vdots\\
0 &\ldots&\vdots&0&\ldots&0 &1  & 0\\
0  &   0  & 0& \ldots &\ldots&\ldots& 0  &1
\end{matrix}\right)=\left(\begin{matrix}
1 &  \alpha_1  & \ldots & \alpha_{n-1}\\
0  &  \\
\vdots & \indent \indent& I_{n-1} \\
0  &
\end{matrix}\right).$$

\end{center}

\end{lemma}

\newpage

\begin{proof}

$$g^{-1}=\left(\begin{matrix}
G_{11} & -G_{21}  & \ldots &  (-1)^{n+1}G_{n1}\\\\
\indent \indent   (-1)^{i+j}G_{ji}
\end{matrix}\right)=\\$$

  $$=\left(\begin{matrix}
G_{11} &  -G_{21}  & \ldots & (-1)^{n+1} G_{n1}\\
0  &  \\
\vdots & \indent \indent g^{-1}|W \\
0  &
\end{matrix}\right).$$

Hence, $$g^{-1} \left(\begin{matrix}
1 &  \alpha_1  & \ldots& \ldots & &&\ldots&\alpha_{n-1}\\
0  &  1 & 0 &\ldots &\ldots &\ldots&&0\\
0 &  0 &1  & 0&\ldots&\ldots&0&0&\\
\vdots& \vdots& 0& &&&\vdots&\vdots\\
0 &\ldots&\vdots&0&\ldots&0 &1  & 0\\
0  &   0  & 0& \ldots &\ldots&\ldots& 0  &1
\end{matrix}\right)g=g^{-1}hg=\\$$
$$\left(\begin{matrix}
G_{11} &  -G_{21}  & \ldots & (-1)^{n+1} G_{n1}\\
0  &  \\
\vdots & \indent \indent g^{-1}|W \\
0  &
\end{matrix}\right){\Large h}\left(\begin{matrix}
g_{11} &  g_{12}  & \ldots & g_{1n}\\
0  &  \\
\vdots & \indent \indent g|W \\
0  &
\end{matrix}\right)=\\$$
$$\left(\begin{matrix}
G_{11} &  -G_{21}  & \ldots & (-1)^{n+1} G_{n1}\\
0  &  \\
\vdots & \indent \indent g^{-1}|W \\
0  &
\end{matrix}\right)\left(\begin{matrix}
g_{11}, &  (g_{12},...,g_{1n})+(\alpha_1,...,\alpha_{n-1})g|W  \\
0  &  \\
\vdots & \indent  g|W \\
0  &
\end{matrix}\right)=\\$$
$$\left(\begin{matrix}
G_{11} &  -G_{21}  & \ldots & (-1)^{n+1} G_{n1}\\
0  &  \\
\vdots & \indent \indent g^{-1}|W \\
0  &
\end{matrix}\right)\left(\begin{matrix}
g_{11} &  g_{12}  & \ldots & g_{1n}\\
0  &  \\
\vdots & \indent \indent g|W \\
0  &
\end{matrix}\right)+\\$$

$$\left(\begin{matrix}
G_{11} &  -G_{21}  & \ldots & (-1)^{n+1} G_{n1}\\
0  &  \\
\vdots & \indent \indent g^{-1}|W \\
0  &
\end{matrix}\right)\left(\begin{matrix}
0 & \indent \indent (\alpha_1,...,\alpha_{n-1})g|W \\
0  &  0 & \ldots & 0\\
\vdots & \vdots & \ddots & \vdots\\
0  &   0  & \ldots     &0
\end{matrix}\right)=\\$$
$$I+\left(\begin{matrix}
0 & \indent \indent G_{11}(\alpha_1,...,\alpha_{n-1})g|W \\
0  &  0 & \ldots & 0\\
\vdots & \vdots & \ddots & \vdots\\
0  &   0  & \ldots     &0
\end{matrix}\right),$$ as claimed.

\end{proof}

\begin{corollary}\label{FC1}
Suppose $g\in G$ is a transvection.Then 

$$g^{-1}hg=
\left(\begin{matrix}
1 &   \indent \indent \indent  (\alpha_1,...,\alpha_{n-1})g|W \\
0  &  1 & 0 &\ldots &\ldots &\ldots&&0\\
0 &  0 &1  & 0&\ldots&\ldots&0&0&\\
\vdots& \vdots& 0& \vdots&&&\vdots&\vdots\\
0 &\ldots&\vdots&0&\ldots&0 &1  & 0\\
0  &   0  & 0& \ldots &\ldots&\ldots& 0  &1
\end{matrix}\right)=\left(\begin{matrix}
1 &  & (\alpha_1,...,\alpha_{n-1})g|W  & \\
0  &  \\
\vdots & \indent \indent& I_{n-1} \\
0  &
\end{matrix}\right)$$

\end{corollary}

\begin{proof}
Since $g$ is transvection we have $g\in SL(V).$ Now $g|W$ is either the identity or a transvection on $W.$ Hence in the $1^{st}$ case there is nothing to prove and in the $2^{nd}$  case $det(g|W)=1$ and $g_{11}=1.$
Since $G_{11}=det(g|W),$ the result follows from  \textbf{Lemma \ref{FL1}}.
\end{proof}
$$\\$$
Let $V$ be a $G-$module and $W$ a $G-$submodule. Recall that $Fix_G(W):=\{g\in G|g|W=id\}.$

\begin{proposition}\label{FP1}
Suppose $G$ is generated by transvections and $dim_FW=dim_FV-1.$ Then
$S(V)^{Fix_{G}(W)}=S(W)[z],$ where
$z=v^{p^t}+q_1v^{p^{t-1}}+...+q_tv ,q_i\in S(W)^{G|W}$ is  homogenous of
degree $p^t-p^{t-i}$  and $|Fix_{G}(W)|=p^t.$
\end{proposition}

\begin{proof}
We fix a basis $\{v,w_{n-1},...,w_1\}$ of $V$ and $\{w_{n-1},...,w_1\}$ a basis of $W.$  The action of all elements of $G$ are given by matrices with respect to this
basis. So $Fix_{G}(W)=<\sigma_1,...,\sigma_t>,$ where $\{\sigma_1,...,\sigma_t\}$ is a minimal generating set and  $$  \sigma_i=
\left(\begin{matrix}
1 &  \alpha_1^{(i)}  & \ldots& \ldots & &&\ldots&\alpha_{n-1}^{(i)}\\
0  &  1 & 0 &\ldots &\ldots &\ldots&&0\\
0 &  0 &1  & 0&\ldots&\ldots&0&0&\\
\vdots& \vdots& 0& &&&\vdots&\vdots\\
0 &\ldots&\vdots&0&\ldots&0 &1  & 0\\
0  &   0  & 0& \ldots &\ldots&\ldots& 0  &1
\end{matrix}\right)=\left(\begin{matrix}
1 &  \alpha_1^{(i)}  & \ldots & \alpha_{n-1}^{(i)}\\
0  &  \\
\vdots & \indent \indent& I_{n-1} \\
0  &
\end{matrix}\right),i=1,...,t.$$

Hence
$\{(\alpha_1^{(i)},...,\alpha_{n-1}^{(i)})|i=1,...,t\}$ is a
$\mathbb{F}_p-$linearly independent set. Since, by \textbf{Corollary \ref{FC1}} , $Fix_{G}(W)$ is a normal subgroup of $G,$ then for a transvection $g\in G ,  \{g^{-1}\sigma_1g,...,g^{-1}\sigma_tg\}$ is also a minimal
generating set of $Fix_{G}(W).$ Consequently,\\ $span_{\mathbb{F}_p}\{(\alpha_1^{(i)},...,\alpha_{n-1}^{(i)})|i=1,...,t\}=span_{\mathbb{F}_p}\{(\alpha_1^{(i)},...,\alpha_{n-1}^{(i)})(g|W)|i=1,...,t\}.\\$

Recall that $S(V)^{Fix_{G}(W)}=S(W)[z],$where $z$ is obtained
via the following procedure:\\
\[\left\{
\begin{array}{lll}
z_1=v^p-(\sigma_1-I)(v)^{p-1}v=v^p-(\alpha_1^{(1)}w_{n-1}+...+\alpha_{n-1}^{(1)}w_1)^{p-1}v. \\
z_i=z_{i-1}^p-(\sigma_i-I)(z_{i-1})^{p-1}z_{i-1}, i=2,...,t. \\
z=z_t .

 \end{array}
  \right.
  \]

Similarly, set $$
\sigma_{i,g}=\left(\begin{matrix}
1 &   \indent \indent \indent  (\alpha_1^{(i)},...,\alpha_{n-1}^{(i)})g|W \\
0  &  1 & 0 &\ldots &\ldots &\ldots&&0\\
0 &  0 &1  & 0&\ldots&\ldots&0&0&\\
\vdots& \vdots& 0& \vdots&&&\vdots&\vdots\\
0 &\ldots&\vdots&0&\ldots&0 &1  & 0\\
0  &   0  & 0& \ldots &\ldots&\ldots& 0  &1
\end{matrix}\right)=\left(\begin{matrix}
1 &  & (\alpha_1^{(i)},...,\alpha_{n-1}^{(i)})g|W  & \\
0  &  \\
\vdots & \indent \indent& I_{n-1} \\
0  &
\end{matrix}\right),$$

 for $i=1,...,t.$ Hence
$z_1^{'}=v^p-(\sigma_{1,g}-I)(v)^{p-1}v=v^p-(\beta_1^{(1)}w_{n-1}+...+\beta_{n-1}^{(1)}w_1)^{p-1}v,$
where
$(\beta_1^{(i)},...,\beta_{n-1}^{(i)})=(\alpha_1^{(i)},...,\alpha_{n-1}^{(i)})(g|W).$
Hence 
$g((\sigma_i-I)(v))=g(\alpha_1^{(i)}w_{n-1}+...+\alpha_{n-1}^{(i)}w_1)=\beta_1^{(i)}w_{n-1}+...+\beta_{n-1}^{(i)}w_{1},$
for $i=1,...,t.$ If
$z_i^{'}=(z_{i-1}^{'})^p-[(\sigma_{i,g}-I)(z_{i-1}^{'})]^{p-1}z_{i-1}^{'}$ and $z^{'}=z_t^{'},$
then, by induction, starting with $i=1,$ we have :\\

\[\left\{
\begin{array}{lll}
z_i=v^{p^i}+\gamma_{1i}v^{p^{i-1}}+...+\gamma_{ii}v,\gamma_{ji}\in
S(W), 1\le j\le i. \\
z_i^{'}=v^{p^i}+\delta_{1i}v^{p^{i-1}}+...+\delta_{ii}v,\delta_{ji}\in
S(W), 1\le j\le i. \\
g(\gamma_{ji})=\delta_{ji},\forall j=1,...,i.

 \end{array}
  \right.
  \]

Since $S(V)^{Fix_{G}(W)}=S(W)[z]=S(W)[z^{'}]$ and $z,z^{'}$ are monic polynomials in $v$ with $0$ as the free coefficient, it follows that $
z^{'}=z$ and consequently $\delta_{jt}=\gamma_{jt},$ for $j=1,..,t.$ Therefore 
$g(\gamma_{jt})=\gamma_{jt}$ for $1\le j\le t,$ and all transvections $g\in G.$
Hence $q_j:=\gamma_{jt}\in S(W)^{G|W}$ for $j=1,...,t.$
\end{proof}

$$\\$$

Let $V$ be a $F-$finite dimensional $G$-module,then $V^{*}=Hom_F(V,F)$ is a $G-$module by the action $(g\cdot f)(v)=f(g^{-1}v),$ for $f\in V^{*}, g\in G $ and $v\in V.$ 
Let $W\subset V$ be a $G-$submodule. Set $W^{\bot}=\{f\in V^{*}|f(W)=0  \}.$
Clearly $W^{\bot}$ is a $G-$submodule of $V^{*}$ and $dim_FW^{\bot}=dim_FV-dim_FW.$

\begin{lemma}\label{FL2}
$V^{*}/W^{\bot}\cong W^{*} $ as $G-$modules.

\end{lemma}

\begin{proof}
Consider the map $\pi:V^{*} \longrightarrow W^{*}$ given by $\pi(f)=f|W,$ that is the restriction of $f$ to $W. Ker (\pi)=\{f\in V^{*}|f|W=0\}=W^{\bot}.$
Hence the induced map $\bar{\pi}:V^{*}/W^{\bot} \longrightarrow W^{*}$ is injective.
Since $dim_F V^{*}-dim_F W^{\bot}=dim_F V-dim_F W^{\bot}=dim_F W=dim_F W^{*}$
we get that $\bar{\pi}$ is an isomorphism. It is also easy to cheak that $\pi$ and hence $\bar{\pi}$ are $G-$module maps.
\end{proof}

\newpage

We also have the following results.

\begin{corollary}\label{FC2}
$V/W\cong (W^{\bot})^{*}.$

\end{corollary}

\begin{proof}
$(W^{\bot})^{\bot}=\{v\in V|g(v)=0, \forall g\in W^{\bot}\}.$ Hence $W\subset (W^{\bot})^{\bot}.$ Since $dim_F (W^{\bot})^{\bot}=dim_F V^{*}-dim_F W^{\bot}=
dim_F V^{*}-(dim_F V-dim_FW)=dim_FW$ we get that $(W^{\bot})^{\bot}=W.$
We now apply \textbf{Lemma \ref{FL2}}  to the pair $W^{\bot}\subset V^{*}.$
Hence $V/W=(V^{*})^{*}/(W^{\bot})^{\bot}\cong (W^{\bot})^{*}.$ 

\end{proof}

\begin{corollary}\label{FC4}
$\{g\in G|g(V)\subseteq W\}=\{g\in G|(g|W^{\bot})=0\}.$

\end{corollary}

\begin{proof}
This is a consequence of  \textbf{Corollary \ref{FC2}}, since $\{g\in G|g(V)\subseteq W\}=Ann_GV/W=Ann_G(W^{\bot})^{*}=Ann_GW^{\bot}=\{g\in G|g|W^{\bot}=0\}.$
\end{proof}

\begin{corollary}\label{FC3}
Let $W\subset V$ be a $G-$submodule. Then 
\begin{enumerate}
\item $W$ is faithful $G-$module iff $W^{*}$ is a faithful $G-$module iff $V^{*}/W^{\bot}$ is a faithful $G-$module.
 \item $W$ is an irreducible $G-$module iff $W^{*}$ is an irreducible $G-$module iff
 $V^{*}/W^{\bot}$ is an irreducible $G-$module.
\end{enumerate}

\end{corollary}

\begin{proof}
This follows from \textbf{Lemma \ref{FL2}} .
\end{proof}

\newpage

\section{Chapter($D$):  $V$ is  indecomposable,contains a unique $2-$ dimensional $G-$ submodule, but no
$1-$ dimensional $G-$ submodule.}
$\indent$In this chapter we consider $V,$ a reducible and
indecomposable $G-module$ which has a unique  $G$-submodule $W$
with $dim_FW=2,$ and $V$ contains no $1-$dimensional $G-$submodules,
hence $W$ is an irreducible $G-module$.

\begin{lemma}\label{DL1}
One of the following holds:
\begin{enumerate}
\item $W$ is a reducible $T(G)-$ module and $\sigma|W=id , \forall
\sigma\in T(G).$ \item $W$ is an irreducible $T(G)-$ module.

\end{enumerate}

\end{lemma}
\begin{proof}

 Assume that $W$ is a reducible $T(G)-$ module.Then
there exists $w\in W$ such that $Fw$ is a $T(G)-$ module, and it
follows,using \textbf{Lemma \ref{AL2}}, that $\sigma(w)=w$ $\forall
\sigma\in T(G). $ Let $g\in G ,$ since $T(G)\lhd  G ,$
$(g^{-1}\sigma g)=\tau $ for some $\tau \in T(G),$ hence
$\sigma (g(w))=g(\tau(w))=g(w).$ If $g(w)\not\in Fw$ then $\{w,g(w)\}$ is a
basis for $W$ and  $\sigma$ acts trivially on $W$ $\forall \sigma
\in T(G)$ so \textbf{(1)} is established. On the other hand if  $g(w)\in Fw$ $\forall g\in G ,$
then $Fw$ is $1-$dimensional $G-$module.This contradicts the
assumption that $V$ contains no $1-$dimensional $G-$modules.

\end{proof}
For $W$  an irreducible  $T(G)-$ module,
  the transvection group $T(G)|W$  had been classified  by Kantor
whenever $T(G)|W$ is an irreducible primitive linear group, see e.g 
\cite[Theorem 1.5]{Kemper3}.
So  one of the following holds in this case : \\
\begin{enumerate}[label=(\roman*)]
\item $T(G)|W=SL(2,\mathbb{F}_q),$ where $p|q.$
 \item $T(G)|W\cong SL(2,\mathbb{F}_5),T(G)|W\subset SL(2,\mathbb{F}_9)$ and $\mathbb{F}_9\subseteq F.$

 \end{enumerate}
Let $\{v,w_2,w_1\}$ be a basis of $V=F^3$ where $\{w_1,w_2\}$ is
a basis of $W.$ So every $g\in G$  is represented with respect to
this basis, by the matrix:\[(5.a) \left(
\begin{array}{ccc}
\lambda(g)& \delta_{12}(g) & \delta_{13}(g) \\
0 & \delta_{22}(g)& \delta_{23}(g) \\
0 & \delta_{32}(g)& \delta_{33}(g) \end{array} \right).\] We shall
make the following assumption:

\begin{assumption}\label{DA1}
Let $\alpha:T(G)\longrightarrow (T(G)|W)$ be the map 
defined by $\alpha(\tau)=(\tau|W)$ and let $N:=ker(\alpha).$
Then  $S(V)^{T(G)}$ is a polynomial ring,
$S(V)^{T(G)}=S(W)^{T(G)}[n]=F[n,m_1,m_2]$ where $m=(m_1,m_2)$ is
the homogenous maximal ideal of $S(W)^{T(G)}=F[m_1,m_2] ,$ and
$S(V)^{N}=S(W)[n].$
\end{assumption}
\newpage

The next result provides a necessary and sufficient condition on $G$
to make  $S(V)^G$ into a  Gorenstein ring .

\begin{proposition}\label{DP1}
We keep  the above notation and \textbf{Assumption \ref{DA1}}. Let $G\subset SL(V)$
be a finite group.Then the following are equivalent:

\begin{enumerate}
\item $S(V)^G$ is Gorenstein. \item
$det(g|m/m^2)=\lambda(g)^{1-d}det(g|W),$ where
$d=deg(n)=\frac{|T(G)|}{deg(m_1)\cdot deg(m_2)}$ and
$m=(m_1,m_2).$

\end{enumerate}

\end{proposition}

\begin{proof}
 Suppose firstly that $S(V)^G$ is Gorenstein.Then by [Theorem
\ref{T2}]  $det(g|p/p^2)=1,$
 where $p=(n,m_1,m_2)$ is the homogenous maximal ideal of $S(V)^{T(G)}.$ Clearly $g$ acts linearly on $p/p^2$ and is represented by the matrix:

\[
\left[
\begin{array}{c|c}
\lambda(g)^d & \,* \quad  *\,  \\
\hline \begin{matrix} 0\\ 0\end{matrix}& g|m/m^{2}
\end{array}
\right]
\]
 So
$det(g|p/p^2)=\lambda(g)^ddet(g|m/m^2)=1.$ Since $G\subset SL(V),
\lambda(g)det(g|W)=1.$
Consequently,\\ $\lambda(g)^ddet(g|m/m^2)=\lambda(g)det(g|W),$ hence
$det(g|m/m^2)=\lambda(g)^{1-d}det(g|W).$ Suppose now that
$det(g|m/m^2)=\lambda(g)^{1-d}det(g|W).$ Hence,using $G\subset
SL(V),$ $\lambda(g)^ddet(g|m/m^2)=\lambda(g)det(g|W)=1.$ So
$det(g|p/p^2)=1 ,$ then by [Theorem \ref{T2}] $S(V)^G$ is Gorenstein .

\end{proof}

\begin{theorem}\label{DT2}
Suppose $W\subset V$ is an irreducible $G-$submodule and $\sigma|W=Id,$ for all $\sigma \in T(G)$ (as in Lemma \ref{DL1} (1)). Then $S(V)^G$ is Gorenstein iff $\lambda(g)^{1-d}=1.$
Consequently if $V=\mathbb{F}_p^3$ then $S(V)^G$ is always Gorenstein.
\end{theorem}

\begin{proof}
By assumption $T(G)=ker(\alpha),$ where  $\alpha:T(G)\longrightarrow (T(G)|W)$
is the restriction map. Therefore by \textbf{Proposition \ref{FP1}} $S(V)^{T(G)}=S(W)[n],$ where $d=deg(n)=p^t.$ Hence \textbf{Assumption \ref{DA1}} holds here with $m_1=w_1,m_2=w_2,m=(w_1,w_2),p=(n,w_1,w_2).$ Since $m/m^2=Fw_1+Fw_2=W$ we have that $det(g|m/m^2)=det(g|W),$ implying by \textbf{Proposition \ref{DP1}} that $S(V)^G$ 
is  Gorenstein if and pnly if $\lambda(g)^{1-d}=1 .$ Consequently, if $V=\mathbb{F}_p^3$ then , since 
$\lambda(g)\in \mathbb{F}_p ,\lambda(g)^{1-d}=1 .$

\end{proof}

The next result applies in case \textbf{(i)}, where $W$ is an irreducible primitive $T(G)-$module.\\

 \begin{proposition}\label{DP2}

Suppose  $G\subset SL(V)$ is a finite group with
$T(G)|W=SL(2,\mathbb{F}_q)$ ,where $q=p^s ,$ 
$\mathbb{F}_q\subseteq F$ and satisfies \textbf{Assumption \ref{DA1}}. Let $V=F^3 $ be the natural
$SL(3,F)-module .$Let $g\in G$ be arbitrary and set $g|W=\bigl(\begin{smallmatrix} a&b \\
c&d
\end{smallmatrix} \bigr)$ with respect to the $\mathbb{F}_q-$ basis $\{w_1,w_2\}.$Then $S(V)^{G}$ is Gorenstein if and only if
$a^{q^2-1}=\lambda(g)^{1-d}.$

\end{proposition}

\begin{proof}
The computation in \textbf{Proposition \ref{APb}}  shows that the matrix representing $g|m/m^2$ is  $\bigl(\begin{smallmatrix} a^{q-1}det(g|W)&0 \\
0&a^{q^2-q}
\end{smallmatrix}\bigr),$ hence
$det(g|m/m^2)=a^{q^2-1}det(g|W).$ Hence,by \textbf{Proposition \ref{DP1}}, $S(V)^G$ is Gorenstein if and only if
$a^{q^2-1}det(g|W)=det(g|W)\lambda(g)^{1-d}.$ Thus $S(V)^G$ is
Gorenstein if and only if $a^{q^2-1}=\lambda(g)^{1-d}.$

\end{proof}
\newpage
\begin{remark}\label{remark}

For $W$  an  irreducible and imprimitive   $T(G)-module, T(G)|W$ is a monomial
subgroup (e.g.\cite{ZS}).In this case, if  $charF\neq 2$ , then by \cite[Lemma 3.9.]{Braun3} $T(G)|W={1} ,$ which  contradicts  the assumption that $W$ is an irreducible  $T(G)-$module.
On the other hand, if  $p=2,$ then $T(G)|W=<\sigma>$,where $\sigma=\bigl(\begin{smallmatrix} 0&1 \\ 1&0
\end{smallmatrix} \bigr).$ This implies that $(w_2+w_1)\in S(V)^{G}.$ Hence $\mathbb{F}_2(w_2+w_1)$
is $G-$module,which  contradicts  the assumption that $W$ is an irreducible  $G-$module.Hence we deal only with the case when $W$ is an irreducible primitive $T(G)-$module.

\end{remark}

\begin{proposition}\label{DP3}
Suppose $F=\mathbb{F}_p$ and $W$ is an irreducible   $T(G)-$module. Then \textbf{Assumption \ref{DA1}} holds for $ G.$ 
\end{proposition}

\begin{proof}
Clearly \[N=ker(\alpha)\subseteq
 \left(
\begin{array}{ccc}
1& \mathbb{F}_p& \mathbb{F}_p\\
0& 1& 0\\
 0 & 0 & 1\end{array} \right).\] 
 Since $F=\mathbb{F}_p,$ we have by \textbf{Remark  \ref{remark}}  that $W$ is a primitive $T(G)-$module, and therefore  $T(G)|W=SL(2,\mathbb{F}_p).$
 If $N\neq 1$ then there exists \[\sigma:=
 \left(
\begin{array}{ccc}
1& \alpha & \beta\\
0& 1& 0\\
 0 & 0 & 1\end{array} \right),(\alpha,\beta)\neq (0,0),\sigma\in N.\]
Consequently by \textbf{Corollary \ref{FC1}} we have :
 \[
\left(
\begin{array}{ccc}
1&\indent & (\alpha,\beta)SL(2,\mathbb{F}_p)\\
0& 1& 0\\
 0 & 0 & 1\end{array} \right)\subseteq N,\] and therefore since 
 $(\alpha,\beta)SL(2,\mathbb{F}_p)=(\mathbb{F}_p,\mathbb{F}_p)$ 
 we conclude that  \[N=
 \left(
\begin{array}{ccc}
1& \mathbb{F}_p& \mathbb{F}_p\\
0& 1& 0\\
 0 & 0 & 1\end{array} \right).\]
 Let \[g=
 \left(
\begin{array}{ccc}
1& \delta_{12}(g)& \delta_{13}(g)\\
0& \delta_{22}(g)& \delta_{23}(g)\\
 0 & \delta_{32}(g) & \delta_{33}(g)\end{array} \right)\in T(G),\] be an arbitrary element.Then \[
 \left(
\begin{array}{ccc}
1& \delta_{12}(g)& \delta_{13}(g)\\
0& \delta_{22}(g)& \delta_{23}(g)\\
 0 & \delta_{32}(g) & \delta_{33}(g)\end{array} \right)
 \left(
\begin{array}{ccc}
1& -\delta_{12}(g)& -\delta_{13}(g)\\
0& 1& 0\\
 0 & 0 & 1\end{array} \right)=\left(
\begin{array}{ccc}
1& 0& 0\\
0& \delta_{22}(g)& \delta_{23}(g)\\
 0 & \delta_{32}(g) & \delta_{33}(g)\end{array} \right)\in T(G)N=T(G).\]
Consequently,since $T(G)|W=\{g|W=\bigl(\begin{smallmatrix} \delta_{22}(g)&\delta_{23}(g) \\
\delta_{32}(g)&\delta_{33}(g)
\end{smallmatrix} \bigr)|g\in T(G)\}$ and $SL(2,\mathbb{F}_p)=T(G)|W,$ then \[H:=\left(
\begin{array}{c|c}
1& \,0 \quad  0\, \\
\hline\begin{matrix} 0\\ 0\end{matrix} &SL(2,\mathbb{F}_p)
\end{array}
\right)\subseteq T(G).\]
Moreover $H\cap N=\{1\}.$Hence $|HN|=|H|\cdot|N|=(p^2-p)\cdot(p+1)\cdot|N|=|T(G)|W|\cdot|N|=
\frac{|T(G)|}{|N|}\cdot|N|=|T(G)|,$ implying that $T(G)=HN.$ Let $\{v,w_2,w_1\}$ be the basis representing the matrices above,where $W=span_{\mathbb{F}_p}\{w_1,w_2\}.$
Now by \textbf{Proposition \ref{FP1}} , there exists $n=v^{p^2}+q_1v^p+q_2v, q_i\in S(W)^{T(G)|W},$
with $S(V)^N=S(W)[n].$ Since $h(v)=v$ for each $h\in H,$ it follows that $n\in S(V)^H,$
and therefore $S(V)^{T(G)}=S(W)[n]^{T(G)/N}=S(W)^{H}[n]=S(W)^{SL(2,\mathbb{F}_p)}[n]=F[d_1,d_2,n],$
where $\{d_1,d_2\}$ are the Dickson invariants of degrees $\{p+1,p^2-p\}$ respectively.\\
Suppose $N=\{1\}.$ Hence $T(G)\cong T(G)|W=SL(2,\mathbb{F}_p).$ Now $T(G)=<\phi_1,...,\phi_k>,$ where $\phi_i,$ is a transvection, for $i=1,...,k.$
Moreover $k\geq 2$ (since $T(G)\cong SL(2,\mathbb{F}_p)$), and we may assume that at least two out of $\{\phi_1|W,...,\phi_k|W\}$ do not commute, say $\{\phi_1|W,\phi_2|W\}.$
Therefore $<\phi_1|W,\phi_2|W>\cong SL(2,\mathbb{F}_p),$ implying that $\phi_i|W\in <\phi_1|W,\phi_2|W>, $ for each $i\geq 3.$ Therefore, since $\alpha$ is injective, we have: $\phi_i\in <\phi_1,\phi_2>,$ for $i\geq 3,$ that is $T(G)=<\phi_1,\phi_2>.$ Hence $ker(\phi_1-I)\cap ker(\phi_2-I):=Fv,$
is a one dimensional trivial $T(G)-$module. Moreover, since $W$ is an irrreducible $T(G)-$module, $v\not\in W,$ hence $V=Fv\oplus W$ as $T(G)-$modules. Consequently with respect to the basis $\{v,w_1,w_2\}$ (where $\{w_1,w_2\}$ is any basis of $W$), $T(G)$
is represent by matrices of the form : 
 \[H:=\left(
\begin{array}{c|c}
1& \,0 \quad  0\, \\
\hline\begin{matrix} 0\\ 0\end{matrix} &SL(2,\mathbb{F}_p)
\end{array}
\right),\] and since $|H|=|SL(2,\mathbb{F}_p)|=|T(G)|W|=|T(G)|$ we have $H=T(G).$
 Hence $S(V)^{T(G)}=S(W)^{SL(2,\mathbb{F}_p)}[v]=F[d_1,d_2,v]$ and we take $n=v.$

\end{proof}

\begin{theorem}\label{DT1}
Suppose $G\subset SL(3,\mathbb{F}_p)$ and  $W$ is an irreducible   $T(G)-$submodule of $V=\mathbb{F}_p^{3}.$ Then $S(V)^G$ is Gorenstein.
\end{theorem}

\begin{proof}
By \textbf{Proposition \ref{DP2}} and \textbf{Proposition \ref{DP3}}, we have to check that $a^{p^2-1}=\lambda(g)^{1-d},$
where $d=deg(n)\in \{1,p^2\}$ and $g|W=\bigl(\begin{smallmatrix} a&b \\
c&d
\end{smallmatrix}\bigr).$ If $N=1$ then $d=1 ,$ and if $N\neq 1 ,$ then $d=p^2 .$ Since $a,\lambda(g)\in \mathbb{F}_p $ we have $a^{p-1}=\lambda(g)^{p-1}=1.$
Hence $a^{p^2-1}=1=\lambda(g)^{1-d}$ for  $d\in \{1,p^2\}.$ 
\end{proof}

\begin{corollary}\label{DC1}
Suppose $G\subset SL(3,\mathbb{F}_p),W\subset V=\mathbb{F}_p^3$ is an irreducible $G-$submodule and $dim_{\mathbb{F}_p}W=2.$ Then $S(V)^G$ is Gorenstein.

\end{corollary}

\begin{proof}
This is a combination of \textbf{Lemma \ref{DL1}}, \textbf{Theorem \ref{DT2}} and  
\textbf{Theorem \ref{DT1}} 
\end{proof}

\newpage

\section{Chapter($E$): $V$ is  indecomposable,contains a unique $1-$ dimensional $G-$ submodule $U,$ but no $2-$ dimensional $G-$ submodule.}
$\indent$In this chapter we consider  $V,$ a $3-$dimensional
reducible and indecomposable $G-module$ which has a unique
$1-$dimensional $G$-submodule $U=Fv_0$ and $V$ contains no
$2-$dimensional $G-$submodules.Hence  $W=V/U$ is an irreducible
$G-$ module. We shall exclusively consider here the case $F=\mathbb{F}_p.\\$ Let $\{w_1,w_2,v_0\}$ be any basis
of $V=F^3.$ So every $g\in G$ is represented with
respect to this basis, by the matrix:\[(6.a) \indent
 g=\left(
\begin{array}{ccc}
\delta_{11}(g)& \delta_{12}(g)& \delta_{13}(g)\\
\delta_{21}(g)& \delta_{22}(g)& \delta_{23}(g)\\
0 & 0 & \lambda(g)\end{array} \right).\]

Let $\alpha:T(G)\longrightarrow (T(G)|W)$ be the map which is
defined by $\alpha(\tau)=\tau|W, $ hence $ker(\alpha)=\{\sigma\in
T(G)|\sigma(x)-x\in Fv_0, \forall x\in V\}.$Consequently,
\[ker(\alpha)\subseteq \left(
\begin{array}{ccc}
1& 0& *\\
0& 1& *\\
0 & 0 & 1\end{array} \right).\]

We shall firstly consider the case that $W=V/\mathbb{F}_pv_0$ is an irreducible $T(G)-module$ (on top of being $G-$irreducible).Since $F=\mathbb{F}_p,$it follows from \textbf{Remark  \ref{remark}},  that $T(G)|W=SL(2,\mathbb{F}_p).$

\begin{proposition}\label{EP1}
Suppose $V=\mathbb{F}_p^3,$ and $V/\mathbb{F}_pv_0$ is an
unfaithful, irreducible $T(G)-$module.Then $S(V)^{T(G)}$ is a
polynomial ring.

\end{proposition}

\begin{proof}
Clearly $T(G)\subset SL(V)$.Let $N:=ker (\alpha),$ then by
assumption $N\neq 1.$ Set $U=\mathbb{F}_pv_0.$ Then by \textbf{Corollary \ref{FC4}} , $N=Ann_{T(G)}V/U=Ann_{T(G)}U^{\bot}.\\ V/U$ is an irreducible $T(G)-$module hence by 
\textbf{Corollary \ref{FC3}} , $U^{\bot}$ is a $2-$dimensional irreducible $T(G)-$ submodule of $V^{*}.$ Therefore by the proof of the early part of \textbf{Proposition \ref{DP3}} , $|N|=p^2.$ 

$$\\$$ Hence,
\[N=\left(
\begin{array}{ccc}
1& 0& \mathbb{F}_p\\
0& 1& \mathbb{F}_p\\
0 & 0 & 1\end{array} \right),\]with generators \[\{\left(
\begin{array}{ccc}
1& 0& 1\\
0& 1& 0\\
0 & 0 & 1\end{array} \right),\left(
\begin{array}{ccc}
1& 0& 0\\
0& 1& 1\\
0 & 0 & 1\end{array} \right)\}.\]

\newpage

 Let
\[g=\left(
\begin{array}{ccc}
\delta_{11}(g)& \delta_{12}(g)& \delta_{13}(g)\\
\delta_{21}(g)& \delta_{22}(g)& \delta_{23}(g)\\
0 & 0 & 1\end{array} \right),\] be a transvection in $T(G).$Then
\[\left(
\begin{array}{ccc}
1& 0& -\delta_{13}(g)\\
0& 1& -\delta_{23}(g)\\
0 & 0 & 1\end{array} \right)\left(
\begin{array}{ccc}
\delta_{11}(g)& \delta_{12}(g)& \delta_{13}(g)\\
\delta_{21}(g)& \delta_{22}(g)& \delta_{23}(g)\\
0 & 0 & 1\end{array} \right)=\left(
\begin{array}{ccc}
\delta_{11}(g)& \delta_{12}(g)& 0\\
\delta_{21}(g)& \delta_{22}(g)& 0\\
0 & 0 & 1\end{array} \right)\in NT(G)=T(G).\]
\\
Since $\{\bigl(\begin{smallmatrix} \delta_{11}(g)&\delta_{12}(g) \\
\delta_{21}(g)&\delta_{22}(g)
\end{smallmatrix} \bigr)|g\in T(G)\}=SL(2,\mathbb{F}_p)=T(G)|W$ then
\[H:=
\left(
\begin{array}{c|c}
SL(2,\mathbb{F}_p)& \begin{matrix} 0\\ 0\end{matrix}\\
\hline\begin{matrix} 0 & 0\end{matrix} &1
\end{array}
\right)\subseteq T(G) .\]

 Since $|T(G)|=|SL(2,\mathbb{F}_p)|\cdot |N|=(p+1)(p^2-p)p^2,$
 then $T(G)=<H,N>.$
  We have  $S(V)^N=\mathbb{F}_p[v_0,a,b]$ where $a=w_1^p-v_0^{p-1}w_1$
 and $b=w_2^p-v_0^{p-1}w_2.$ Consider $\{\sigma,\tau \},$
 generators of $H,$ where
\[\tau=\left(
\begin{array}{ccc}
1& 1& 0\\
0& 1& 0\\
0 & 0 & 1\end{array} \right) , \sigma=\left(
\begin{array}{ccc}
1& 0& 0\\
1& 1& 0\\
0 & 0 & 1\end{array} \right).\] Then
$\tau(b)=b,\tau(a)=a+b,\sigma(a)=a,\sigma(b)=a+b.$ Hence
$A=a^pb-b^pa$ and $B=\frac{a^{p^2}b-b^{p^2}a}{A},$ are Dickson
invariants of $({S(V)^N})^{H}=S(V)^{T(G)}.$So $deg(A)\cdot
deg(B)\cdot  deg(v_0)=p(p+1)\cdot p(p^2-p)=|T(G)|.$ Hence by
\cite[Theorem 3.7.5]{Kemper1} $S(V)^{T(G)}=\mathbb{F}_p[v_0,A,B]$ is a
polynomial ring.

\end{proof}

\begin{lemma}\label{EL1}
Suppose $V=\mathbb{F}_p^3 ,$ and $W=V/\mathbb{F}_pv_0$ is an
irreducible  $T(G)-$module.\\Then one of the following holds.

\begin{enumerate}
\item $W$ is an unfaithful irreducible  $T(G)-$module and
$T(G)=<H,N>.$
 \item $W$ is a faithful irreducible $T(G)-$module and $T(G)=H,$ where $H$ is
defined in  \textbf{Proposition \ref{EP1}}.

\end{enumerate}

\end{lemma}

\begin{proof}
If  item \textbf{(1)} holds then by \textbf{Proposition \ref{EP1}} the result
follows.So we may assume that $W$ is a faithful irreducible
$T(G)-$module.
 Let $T(G)=<\phi_1,...,\phi_k>,$
 so $(T(G)|W)=<\phi_1|W,...,\phi_k|W>.$
If $\phi_i|W=Id$ then $\phi_i\in N=1.$ Hence $\phi_i|W$ are non trivial transvections for
$i=1,...,k.$ Since $W$ is an irreducible $T(G)-$module then there
exist at least two transvections that do not commute ,say
$(\phi_1|W)$ and $(\phi_2|W).$ But then $W$ is an irreducible 
$<\phi_1|W,\phi_2|W>-$module.Hence $<\phi_1|W,\phi_2|W>\cong
SL(2,\mathbb{F}_p).$Since $|(T(G)|W)|=|SL(2,\mathbb{F}_p)|$ then
$<\phi_1|W,\phi_2|W>= (T(G)|W).$ Therefore, $\forall \phi_i\in
T(G)$ with $i\geq 3,(\phi_i|W)$ is generated by $(\phi_1|W)$ and
$(\phi_2|W).$ So $\phi_i=h\cdot Z$ where $h\in <\phi_1,\phi_2>$
and $Z\in N=ker(\alpha).$ Since $N=1,\phi_i=h.$ Therefore
$T(G)=<\phi_1,\phi_2>.\\\\$ 
$\indent$ Consider $U^{\bot}\subset V^{*}.$ Since $W=V/U$ is a faithful and irreducible $T(G)-$module it follows from \textbf{Corollary \ref{FC2}},\textbf{Corollary \ref{FC3}} that $U^{\bot}$
is a faithful and irreducible $T(G)-$submodule of $V^{*}.\\$ 
Consider $ker(\phi_1-I)\cap ker(\phi_2-I)\subset V^{*}.$ It is a $1-$dimensional subspace (since $\phi_1,\phi_2$ are transvections on $V^{*}$). Let $Fg=ker(\phi_1-I)\cap ker(\phi_2-I)\subset V^{*}.$ Since $T(G)=<\phi_1,\phi_2>$ it follows that $Fg$ is a trivial $T(G)-$submodule of $V^{*}.$ Since $U^{\bot}$ is irreducible it follows that $Fg\oplus U^{\bot}=V^{*},$ a direct sum decomposition of $V^{*}$ as $T(G)-$modules.Since $U^{\bot}$ is irreducible and faithful we have , as in \textbf{Proposition \ref{DP3}}, a representation of $T(G)$ as matrices of the form :

\[\left(
\begin{array}{c|c}
1& \,0 \quad  0\, \\
\hline\begin{matrix} 0\\ 0\end{matrix} &SL(2,\mathbb{F}_p)
\end{array}
\right),\] with respect to the basis $\{g,h_1,h_2\},$ where $\{h_1,h_2\}$ is  any basis of $U^{\bot}.$ Let $\{x,w_1,w_2\}$ be the dual basis of $\{g,h_1,h_2\}.$ So $g(x)=1,h_1(x)=h_2(x)=0,h_i(w_j)=\delta_{ij}$ and $g(w_1)=g(w_2)=0.$ Hence $x\in (Fh_1+Fh_2)^{\bot}=(U^{\bot})^{\bot}=U$ and $\mathbb{F}_px=\mathbb{F}_pv_0.\\$
$(Fg)^{\bot}=Fw_1+Fw_2$ is a $2-$dimensional $T(G)-$submodule of $V$ and $(Fg)^{\bot}\oplus (U^{\bot})^{\bot}=V.$ Hence the matrices representing $T(G)$ with respect to the basis $\{w_1,w_2,v_0\}$ are

\[H:=
\left(
\begin{array}{c|c}
SL(2,\mathbb{F}_p)& \begin{matrix} 0\\ 0\end{matrix}\\
\hline\begin{matrix} 0 & 0\end{matrix} &1
\end{array}
\right).\] \\

\end{proof}

The following theorem deals with the special case,when
$W=V/\mathbb{F}_pv_0$ is an irreducible primitive $T(G)-$module
and $F=\mathbb{F}_p.$ Since $F=\mathbb{F}_p,$ by \cite[Theorem 1.5]{Kemper3}
$(T(G)|W)=SL(2,\mathbb{F}_p)$ and $p\neq 2 .$

\begin{theorem}\label{ET1}
Suppose that $G\subset SL(3,\mathbb{F}_p),W=V/\mathbb{F}_pv_0$ is
an irreducible   $T(G)-$module. Then $S(V)^{G}$ is
Gorenstein.

\end{theorem}

\begin{proof}
Assume firstly that $W$ is a faithful irreducible
$T(G)$-module.By \textbf{Lemma \ref{EL1} (2)},
$S(V)^{T(G)}=S(V)^H=\mathbb{F}_p[u,y,v_0]$ is a polynomial ring
where $u=w_1w_2^p-w_2w_1^p,y=\frac{w_1w_2^{p^2}-w_2w_1^{p^2}}{u}$
are the Dickson's invariants.Now for $g\in G$ since
$gT(G)g^{-1}=T(G),\mathbb{F}_p[u,y,v_0]$ is $g$-stable and since
$g$ preserve degrees we get $g(u)=\alpha u+\beta v_0^{p+1}$ with
$\alpha,\beta \in \mathbb{F}_p.$ By  Fermat' little theorem and
formula ($\textbf{6.a}$)
 we get:\\\\
$g(u)=(\delta_{11}(g)w_1+\delta_{12}(g)w_2+\delta_{13}(g)v_0)(\delta_{21}(g)w_1^p+\delta_{22}(g)w_2^p+\delta_{23}(g)v_0^p)-
(\delta_{21}(g)w_1+\delta_{22}(g)w_2+\delta_{23}(g)v_0)(\delta_{11}(g)w_1^p+\delta_{12}(g)w_2^p+\delta_{13}(g)v_0^p)=
\delta_{11}(g)\delta_{22}(g)(w_1w_2^p-w_2w_1^p)+\delta_{11}(g)\delta_{23}(g)(w_1v_0^p-v_0w_1^p)+
\delta_{12}(g)\delta_{23}(g)(w_2v_0^p-v_0w_2^p)+\delta_{13}(g)\delta_{22}(g)(w_2^pv_0-w_2v_0^p)+
\delta_{12}(g)\delta_{21}(g)(w_2w_1^p-w_1w_2^p)+\delta_{13}(g)\delta_{21}(g)(w_1^pv_0-w_1v_0^p)=
(\delta_{11}(g)\delta_{22}(g)-\delta_{12}(g)\delta_{21}(g))(w_1w_2^p-w_2w_1^p)+(\delta_{11}(g)\delta_{23}(g)-\delta_{13}(g)\delta_{21}(g))
(w_1v_0^p-v_0w_1^p)+(\delta_{12}(g)\delta_{23}(g)-\delta_{13}(g)\delta_{22}(g))(w_2v_0^p-v_0^pw_2).$\\\\
Now by comparing coefficients of $g(u)$ we get: \\

\[\textbf{(6.b)}\indent \left\{
\begin{array}{llll}
 \alpha=\delta_{11}(g)\delta_{22}(g)-\delta_{12}(g)\delta_{21}(g)  \\
 \beta=0 \\
\delta_{11}(g)\delta_{23}(g)-\delta_{13}(g)\delta_{21}(g)=0
\\

\delta_{12}(g)\delta_{23}(g)-\delta_{13}(g)\delta_{22}(g)=0

 \end{array}
  \right.
  \]
\\

Since $G\subset
SL(V),\lambda(g)^{-1}=\delta_{11}(g)\delta_{22}(g)-\delta_{12}(g)\delta_{21}(g),$then
$g(u)=\lambda(g)^{-1}u.$
 Similarly,$g(y)=\gamma y+q(u,v_0),$
where $\gamma \in \mathbb{F}_p $ and $q(u,v_0)$ is a homogenous
polynomial (in $u,v_0$) with $deg(q)=deg(y)=p^2-p.$\\\\By Fermat'
little theorem , formula \textbf{(6.a)} and \textbf{(6.b)}  we
get:\\\\
$g(w_1w_2^{p^2}-w_2w_1^{p^2})=
(\delta_{11}(g)\delta_{22}(g)-\delta_{12}(g)\delta_{21}(g))(w_1w_2^{p^2}-w_2w_1^{p^2})+(\delta_{11}(g)\delta_{23}(g)-\delta_{13}(g)\delta_{21}(g))
(w_1v_0^{p^2}-v_0w_1^{p^2})+(\delta_{12}(g)\delta_{23}(g)-\delta_{13}(g)\delta_{22}(g))(w_2v_0^{p^2}-v_0^{p^2}w_2)=
\lambda(g)^{-1}(w_1w_2^{p^2}-w_2w_1^{p^2})+0+0,$

 using $\delta_{11}(g)\delta_{22}(g)-\delta_{12}(g)\delta_{21}(g)=\lambda(g)^{-1}.$ Since $uy=w_1w_2^{p^2}-w_2w_1^{p^2}$ we get
$g(y)=\frac{g(w_1w_2^{p^2}-w_2w_1^{p^2})}{g(u)}=\frac{\lambda(g)^{-1}(w_1w_2^{p^2}-w_2w_1^{p^2})}{\lambda(g)^{-1}u}=
 \frac{w_1w_2^{p^2}-w_2w_1^{p^2}}{u}=y.$
 So by comparing coefficients we get that $\gamma=1$ and $q(u,v_0)=0.$

$$\\$$ Clearly,$g$ acts linearly on
$\mathbb{F}_pu+\mathbb{F}_py+\mathbb{F}_pv_0$ and is represented
by the matrix :
\[\left(
\begin{array}{ccc}
\lambda(g)^{-1}& 0& 0\\
0& 1& 0\\
0 & 0 & \lambda(g)\end{array} \right).\]

$$\\$$ Therefore, by [Theorem \ref{T2}] $S(V)^G$ is Gorenstein.\\\\

Assume now that $W$ is an unfaithful irreducible
$T(G)-$module.By  \textbf{Proposition \ref{EP1}}
$S(V)^{T(G)}=\mathbb{F}_p[A,B,v_0]$ is a polynomial ring where
$A=a^pb-b^pa,B=\frac{a^{p^2}b-b^{p^{2}}a}{A}$ and
$deg(A)=p^2+p,deg(B)=p(p^2-p).$  We have :\\\\
$A=a^pb-b^pa=(w_1^{p^{2}}-v_0^{p(p-1}w_1^p)(w_2^p-v_0^{p(p-1)}w_2)-(w_2^{p^2}-v_0^{p(p-1)})w_2^p(w_1^p-v_0^{p-1}w_1)=
(w_1^{p^2}w_2^p-w_2^{p^2}w_1^p)-v_0^{p-1}(w_1^{p^2}w_2-w_2^{p^2}w_1)+v_0^{p^2-1}(w_1^pw_2-w_2^pw_1)=
(-u)^p-v_0^{p-1}(w_1^{p^2}w_2-w_2^{p^2}w_1)+v_0^{p^2-1}(-u),$ where
$u=w_1w_2^p-w_2w_1^p.$ We get in the
previous calculation that :\\
\[\left\{
\begin{array}{ll}
 g(u)=\lambda(g)^{-1}u .\\
 g(w_1^{p^2}w_2-w_2^{p^2}w_1)=\lambda(g)^{-1}(w_1^{p^2}w_2-w_2^{p^2}w_1) . \\

 \end{array}
  \right.
  \]
Since $g(v_0^{p^2-1})=(\lambda(g)^{-1})^{p^2-1}v_0^{p^2-1}=v_0^{p^2-1},$ then  $g(A)=\lambda(g)^{-1}A.$ 

\newpage
We also have :\\

$a^{p^2}b-b^{p^2}a=(w_1^{p^3}-v_0^{p^2(p-1)}w_1^{p^2})(w_2^p-v_0^{p-1}w_2)-(w_2^{p^3}v_0^{p^2(p-1)}w_2^{p^2})(w_1-v_0^{p-1}w_1)=
(w_1^{p^3}w_2^{p^2}-w_2^{p^3}w_1^p)-v_0^{(p-1)}(w_1^{p^3}w_2-w_2^{p^3}w_1)+v_0^{p^2(p-1)}(w_2^{p^2}w_1^p-w_2^{p^2}w_2^p)+
v_0^{(p^2+1)(p-1)}(w_1^{p^2}w_2-w_2^{p^2}w_1)=(-u)^{p^2}-v_0^{(p-1)}(w_1^{p^3}w_2-w_2^{p^3}w_1)+v_0^{p^2(p-1)}(u)^p+
v_0^{(p^2+1)(p-1)}(-uy).\\\\$      
 We get in the previous calculation that $g(u)=\lambda(g)u$ and
$g(y)=1.$ On the other hand we have :\\

$g(v_0^{(p-1)}(w_1^{p^3}w_2-w_2^{p^3}w_1))=(\delta_{11}(g)\delta_{22}(g)-\delta_{12}(g)\delta_{21}(g))(w_1^{p^3}w_2-w_2^{p^3}w_1)+
(\delta_{23}(g)\delta_{11}(g)-\delta_{13}(g)\delta_{21}(g))(v_0^{p^3}w_2-w_2^{p^3}v_0)+
(\delta_{22}(g)\delta_{13}(g)-\delta_{12}(g)\delta_{23}(g))(w_1^{p^3}v_0-w_2^{p^3}v_0).$
Hence by $\textbf{(6.b)},
g(v_0^{(p-1)}(w_1^{p^3}w_2-w_2^{p^3}w_1))=\lambda(g)^{-1}(v_0^{(p-1)}(w_1^{p^3}w_2-w_2^{p^3}w_1).\\$
Consequently $g(a^{p^2}b-b^{p^2}a)=\lambda(g)^{-1}(a^{p^2}b-b^{p^2}a),$ hence
$g(B)=\frac{g(a^{p^2}b-b^{p^2}a)}{g(A)}=\frac{\lambda(g)^{-1}(a^{p^2}b-b^{p^2}a)}{\lambda(g)^{-1}A}=B.$ So
$g$ acts linearly on $\mathbb{F}_pA+\mathbb{F}_pB+\mathbb{F}_pv_0$
and is represented by the matrix :\\

\[\left(
\begin{array}{ccc}
\lambda(g)^{-1}& 0& 0\\
0& 1& 0\\
0 & 0 & \lambda(g)\end{array} \right).\] Again,by [Theorem \ref{T2}]
$S(V)^{G}$ is Gorenstein .

\end{proof}

We shall now consider the possibility that $W$ is a reducible $T(G)-$module.
\begin{lemma}\label{ELa}
 
 Suppose $W=V/Fv$ is a reducible $T(G)-$module.Then $W$ is a trivial $T(G)-$module, that is $ker(\alpha)=T(G).$
 
\end{lemma}

 \begin{proof}
 Let $F\bar{w_1}$ be a $T(G)-$ submodule of $W.$ Hence $\phi(\bar{w_1})=\bar{w_1} , \forall \phi \in T(G).$
 Since $W$ is irreducible $G-$submodule,$F\bar{w_1}$ is not a $G-$submodule.Consequently,there exists $g\in G$ such that $g(\bar{w_1})=\bar{w_2}\not\in F\bar{w_1}. $ Since $T(G)\lhd G,\forall\sigma\in T(G), g^{-1}\sigma g=\tau ,$ for some $\tau\in T(G).$Hence $(g^{-1}\sigma g)(\bar{w_1})=\tau(\bar{w_1})=\bar{w_1},$ and therefore $\sigma(g(\bar{w_1}))=g(\bar{w_1}),$ hence $\sigma(\bar{w_2})=\bar{w_2}.$ But $\{\bar{w_1},\bar{w_2}\}$ is a basis of $W.$Consequently $W$ is a trivial $T(G)-$module.

  \end{proof}

  \begin{corollary}\label{ECa}
  
  Suppose $W=V/\mathbb{F}_pv_0$ is a reducible $T(G)-$module, then $S(V)^{T(G)}$ is a polynomial ring.
  
  \end{corollary}
  
  \begin{proof}
  
  \[ker(\alpha)=T(G)\subseteq \left(
\begin{array}{ccc}
1& 0& \mathbb{F}_p\\
0& 1& \mathbb{F}_p\\
0 & 0 & 1\end{array} \right).\] Hence $|T(G)|\in\{p,p^2\}.$ If $|T(G)|=p^2,$ then 

\[T(G)=\left(
\begin{array}{ccc}
1& 0& \mathbb{F}_p\\
0& 1& \mathbb{F}_p\\
0 & 0 & 1\end{array} \right),\] 
 with generators : \[\left(
\begin{array}{ccc}
1& 0& 1\\
0& 1& 0\\
0 & 0 & 1\end{array} \right), \left(
\begin{array}{ccc}
1& 0& 0\\
0& 1& 1\\
0 & 0 & 1\end{array} \right). \] 
 Hence $S(V)^{T(G)}=\mathbb{F}_p[v_0,x,y],$ where $x=w_1^p-v_0^{p-1}w_1,y=w_2^p-v_0^{p-1}w_2.$ If $|T(G)|=p ,$ then $T(G)$ is generated by : \[ \left(
\begin{array}{ccc}
1& 0& \alpha\\
0& 1& \beta\\
0 & 0 & 1\end{array} \right),\alpha,\beta \in \mathbb{F}_p. \] Then as in \textbf{Lemma \ref{CL1}} $S(V)^{T(G)}=\mathbb{F}_p[v_0,\alpha w_2-\beta w_1,z]$ is a polynomial ring and $z=w_1^{p}-(\alpha v_0)^{p-1}w_1.$
\end{proof}

$$\\$$

\begin{corollary}\label{ECb}

Suppose $W=V/\mathbb{F}_pv_0 $ is a reducible $T(G)-$ module and an irreducible $G-$module. Then $|T(G)|=p^2.$  

\end{corollary}
  
  \begin{proof}
  Suppose as in \textbf{Corollary \ref{ECa}}, that $|T(G)|=p,$ so \[T(G)=< \left(
\begin{array}{ccc}
1& 0& \alpha\\
0& 1& \beta\\
0 & 0 & 1\end{array} \right)>.\]
Since $\mathbb{F}_pv_0,\mathbb{F}_p(\alpha w_1-\beta w_2)$ are trivial $T(G)-$ modules we have that  $A=\mathbb{F}_pv_0+\mathbb{F}_p(\alpha w_1-\beta w_2)=ker(\sigma-I)$ for each transvection $\sigma\in T(G).$ Let $g\in G$ and $\sigma$ a transvection in $G.\\$Then $g^{-1}\sigma g=\tau$ for some transvection $\tau.$ Then for each $a\in A$ we have $(g^{-1}\sigma g)(a)=\tau(a)=a$ and hence $\sigma(g(a))=g(a),$ implying that $g(a)\in A$ for each $g\in G.$  Hence $A$ is a $2-$dimensional $G-$submodule of $V.$ This contradicts the assumption on $V.$

  \end{proof}
 
\begin{theorem}\label{ET2}
Suppose $W=V/\mathbb{F}_pv_0 $ is a reducible $T(G)-$ module and an irreducible $G-$module.Then $S(V)^G$ is Gorenstein.
\end{theorem}

\begin{proof}
We follow here the argument in \textbf{Proposition \ref{CL8}} with $\bigl(\begin{smallmatrix} a
\\ b
\end{smallmatrix}\bigr)=\bigl(\begin{smallmatrix} 1
\\ 0
\end{smallmatrix}\bigr),\bigl(\begin{smallmatrix} c
\\ d
\end{smallmatrix}\bigr)=\bigl(\begin{smallmatrix} 0
\\ 1
\end{smallmatrix}\bigr).$ Recall that by \textbf{Corollary \ref{ECa}} and 
\textbf{Corollary \ref{ECb}}, \[T(G)=\left(
\begin{array}{ccc}
1& 0& \mathbb{F}_p\\
0& 1& \mathbb{F}_p\\
0 & 0 & 1\end{array} \right),\] with generators :  \[\left(
\begin{array}{ccc}
1& 0& 1\\
0& 1& 0\\
0 & 0 & 1\end{array} \right), \left(
\begin{array}{ccc}
1& 0& 0\\
0& 1& 1\\
0 & 0 & 1\end{array} \right). \] 
Hence $S(V)^{T(G)}=\mathbb{F}_p[v_0,x,y] ,$ where $x=w_1^p-v_0^{p-1}w_1,y=w_2^p-v_0^{p-1}w_2.$  So if $m=(v_0,x,y)$ then $\bar{g}\in G/T(G)$ is represented on $m/m^2=\mathbb{F}_p\bar{y}+\mathbb{F}_p\bar{x}+\mathbb{F}_pv_0$ by the matrix : 
\[\left(
\begin{array}{ccc}
\alpha& \gamma & 0\\
\varepsilon& \delta& 0\\
0 & 0 & \lambda(g)\end{array} \right),\] 
and $g\in G$ is represented on $V=\mathbb{F}_pw_2+\mathbb{F}_pw_1+\mathbb{F}_pv_0,$
by the matrix : \[\left(
\begin{array}{ccc}
\lambda_2(g)& \delta_{12}(g) & \delta_{13}(g)\\
\delta_{21}(g)& \lambda_1(g)& \delta_{23}(g)\\
0 & 0 & \lambda(g)\end{array} \right).\]
We shall show that $det(\bar{g})=1,$ for each $\bar{g}\in G/T(G) ,$ implying by 
 [Theorem \ref{T2}] that $S(V)^G$ is Gorenstein. Now $g(y)=\alpha y+\gamma x+\beta v_0^p$ and $g(x)=\delta x+\varepsilon y+\chi v_0^p$ (keeping the notation in \textbf{Lemma \ref{CL8}}).
 Since $g(w_2)=\lambda_2(g)w_2+\delta_{12}(g)w_1+\delta_{13}(g)v_0,$ we have $g(y)=g(w_2)^p-(\lambda(g)v_0)^{p-1}g(w_2)=(\lambda_2(g)w_2+\delta_{12}(g)w_1+\delta_{13}(g)v_0)^p-(\lambda(g)v_0)^{p-1}(\lambda_2(g)w_2+\delta_{12}(g)w_1+\delta_{13}(g)v_0)=
(\lambda_2(g)^pw_2^p-\lambda(g)^{p-1}\lambda_2(g)v_0^{p-1}w_2)+(\delta_{12}(g)^pw_1^p-
\lambda(g)^{p-1}\delta_{12}(g)v_0^{p-1}w_1)+(\delta_{13}(g)^p-\lambda(g)^{p-1}\delta_{13}(g))v_0^p.$ $$\\$$ Consequently $\alpha=\lambda_2(g)^p=\lambda(g)^{p-1}\lambda_2(g), \gamma=\delta_{12}(g)^p=\lambda(g)^{p-1}\delta_{12}(g).$
 Consequently,since $\lambda_2(g),\delta_{12}(g)\in \mathbb{F}_p , $
it follows that $\alpha=\lambda_2(g)$ and  $\gamma=\delta_{12}(g).\\\\$
Similarly using $g(w_1)=\delta_{21}(g)w_2+\lambda_1(g)w_1+\delta_{23}(g)v_0,$
we have $g(x)=g(w_1)^p-(\lambda(g)v_0)^{p-1}g(w_1)=(\delta_{21}(g)w_2+\lambda_1(g)w_1+\delta_{23}(g)v_0)^p-(\lambda(g)v_0)^{p-1}(\delta_{21}(g)w_2+\lambda_1(g)w_1+\delta_{23}(g)v_0)=(\delta_{21}(g)^pw_2^p-\lambda(g)^{p-1}\delta_{21}(g)v_0^{p-1}w_2)+
(\lambda_1(g)^pw_1^p-\lambda(g)^{p-1}\lambda_1(g)v_0^{p-1}w_1)+(\delta_{23}(g)^p-\lambda(g)^{p-1}\delta_{23}(g))v_0^p.\\$
Therefore $\varepsilon=\delta_{21}(g)^p,\delta=\lambda_1(g)^p,$ and since $\lambda_1(g),\delta_{21}(g)\in \mathbb{F}_p,$ we have $\varepsilon=\delta_{21}(g),\delta=\lambda_1(g).$ Therefore $det(\bar{g})=(\alpha \delta-\varepsilon \gamma)\lambda(g)=(\lambda_2(g)\lambda_1(g)-\delta_{21}(g)\delta_{12}(g))\lambda(g)=det(g)=1$ (since $G\subset SL(3,\mathbb{F}_p)$). 

\end{proof}

$\\$ As a consequence of \textbf{Theorem \ref{ET1}} and \textbf{Theorem \ref{ET2}} we have :

\begin{corollary}\label{ECc}

Suppose $W=V/\mathbb{F}_pv_0 $ is an irreducible $G-$module. Then $S(V)^G$ is Gorenstein.

\end{corollary}

\newpage

\section{Chapter($F$): $V$ is indecomposable and contains two $G-$ submodules $\mathbb{F}_pv_0,W,$ where $\mathbb{F}_pv_0\subset W\subset V,$ and $dim_{\mathbb{F}_p}W=2.$}

In this chapter we conside $V,$ a reducible indecomposable $G-$module,which has a unique $1-$dimensional $G-$ submodule $Fv_0,$ and a unique $2-$dimensional $G-$ submodule $W,$ and $Fv_0\subset W\subset V.$   
We shall exclusively consider the case $F=\mathbb{F}_p.\\$
Let $\{w_2,w_1,v_0\}$ be a basis of $V=\mathbb{F}_p^3,$ where $W=\mathbb{F}_pw_1+\mathbb{F}_pv_0.$ 
So every $g\in G$ is represented with respect to this basis , by the matrix :

 \[g=\left(
\begin{array}{ccc}
\lambda_2(g)& \delta_{12}(g)& \delta_{13}(g)\\
0& \lambda_1(g)& \delta_{23}(g)\\
0 & 0 & \lambda(g)\end{array} \right).\]
By \textbf{Lemma \ref{AL2}}, $T(G)$ acts trivially on $\mathbb{F}_pv_0,$ and also on $W/\mathbb{F}_pv_0,$ hence 

 \[T(G)\subseteq \left(
\begin{array}{ccc}
1& \mathbb{F}_p& \mathbb{F}_p\\
0& 1& \mathbb{F}_p\\
0 & 0 & 1\end{array} \right).\]
So one of the following three cases holds:

\begin{enumerate}[label=(\roman*)]
\item $|T(G)|=p,$

\item  $|T(G)|=p^2,$

 \item $|T(G)|=p^3.$

\end{enumerate}

\begin{corollary}\label{G1}
If case $\textbf{(i)}$ holds. Then $S(V)^G$ is Gorenstein.
\end{corollary}
\begin{proof}

Since $|T(G)|=p,T(G):=<\sigma>$ or $<\tau>,$ where

 \[\sigma= \left(
\begin{array}{ccc}
1& 0& a\\
0& 1& b\\
0 & 0 & 1\end{array} \right),\tau=\left(
\begin{array}{ccc}
1& c& d\\
0& 1& 0\\
0 & 0 & 1\end{array} \right),a,b,c,d\in \mathbb{F}_p. \] 
Assume firstly that $T(G)=<\sigma>.$ We shall change the basis of $V$ into $\{w_2,aw_1-bw_2,v_0\}.$We have, $\sigma(aw_1-bw_2)=a(w_1+bv_0)-b(w_2+av_0)=aw_1-bw_2.$ Suppose $a\neq 0,$ so $0\neq w_1':= (aw_1-bw_2)\in S(V)^{T(G)}.$
Therefore $\sigma$ is represented with respect to $\{w_2,w_1',v_0\}$  by the matrix:

 \[\sigma= \left(
\begin{array}{ccc}
1& 0& a\\
0& 1& 0\\
0 & 0 & 1\end{array} \right).\]
Consequently, by \cite[Lemma 3.13]{Braun3}
$S(V)^{T(G)}=\mathbb{F}_p[x,w_1',v_0]$ is a polynomial ring, and $x=w_2^p-(av_0)^{p-1}w_2.$ 
Let $m=(x,w_1',v_0)$ be the graded maximal homogenous ideal of $S(V)^{T(G)}.$\\
Since $T(G)\lhd G,\mathbb{F}_p[x,w_1',v_0]$ is $g-$stable and since $g$ preserve degrees we get that $g(\mathbb{F}_pw_1'+\mathbb{F}_pv_0)\in (\mathbb{F}_pw_1'+\mathbb{F}_pv_0)  $ for every $g\in G.$This implies that $(\mathbb{F}_pw_1'+\mathbb{F}_pv_0) $ is $G-$module.Therefore $G$ has a triangular presentation with respect to the new basis $\{w_2,w_1',v_0\}.$Hence we still use the same $\lambda_2(g),\delta_{12}(g),\delta_{13}(g),\lambda_1(g),\delta_{23}(g),$ for the new presentation of $g,$ with respect to the  basis $\{w_2,w_1',v_0\}.$
\newpage
We have:\\

\[(7.a)\left\{
\begin{array}{llll}
g(x)=\lambda_2(g)^pw_2^p+\delta_{12}(g)^p(w_1')^p+\delta_{13}(g)^pv_0^p-(a\lambda(g)v_0)^{p-1}(\lambda_2(g)w_2+\delta_{12}(g)w_1'+\delta_{13}(g)v_0).\\
g(w_2)=\lambda_2(g)w_2+\delta_{12}(g)w_1'+\delta_{13}(g)v_0.\\
g(w_1')=\lambda_1(g)w_1'+\delta_{23}(g)v_0.\\
g(v_0)=\lambda(g)v_0.
\end{array}
  \right.
  \]

  Now since $gT(G)g^{-1}=T(G),\mathbb{F}_p[x,w_1',v_0]$ is $g-$stable and since $g$ preserve degrees we get:
  \[(7.b)\left\{
\begin{array}{l}
g(x)=\alpha_1(w_2^p-(av_0)^{p-1}w_2)+\alpha_2(w_1')^p+\alpha_3v_0^p+$mixed terms in$ \{w_1',v_0\}.\\
 \alpha_1,\alpha_2,\alpha_3\in \mathbb{F}_p.

\end{array}
  \right.
  \]
Hence by comparing coefficient of $w_2^p$  in $g(x),$ we get that $\alpha_1=\lambda_2(g)^p.$ 
So $g$ acts linearly on $m/m^2$ and is represented by the matrix:

 \[ \left(
\begin{array}{ccc}
\lambda_2(g)^p& 0& 0\\
0& \lambda_1(g)& \delta_{23}(g)\\
0 & 0 & \lambda(g)\end{array} \right)= \left(
\begin{array}{ccc}
\lambda_2(g)& 0& 0\\
0& \lambda_1(g)& \delta_{23}(g)\\
0 & 0 & \lambda(g)\end{array} \right).\]
Hence by [Theorem \ref{T2}] $S(V)^G$ is Gorenstein.If $a=0,$ then $S(V)^{T(G)}=\mathbb{F}_p[w_2,x,v_0]$ where $x=w_1^p-(bv_0)^{p-1}w_1.$So we have:\\

  \[(7.c)\left\{
\begin{array}{ll}
g(x)=(\lambda_1(g)w_1+\delta_{23}(g)v_0)^p-(bv_0)^{p-1}(\lambda_1(g)w_1+\delta_{23}(g)v_0).\\
g(x)=\gamma_1x+\gamma_2w_2^p+\gamma_3v_0^p+ $mixed terms in$   \{w_2,v_0\}.\indent\gamma_1,\gamma_2,\gamma_3\in \mathbb{F}_p.\\

\end{array}
  \right.
  \]
Hence by comparing coefficient of $w_1^p$ in both expressions,we get that $\gamma_1=\lambda_1(g)^p=\lambda_1(g).$ Hence by taking $m=(w_2,x,v_0),$ the matrix represinting $g|m/m^2$
with respect to $\bar{w_2},\bar{x},\bar{v_0}$ is 
\[\left(
\begin{array}{ccc}
\lambda_2(g)& \delta_{12}(g)& \delta_{13}(g)\\
0& \lambda_1(g)^p&0 \\
0 & 0 & \lambda(g)\end{array} \right),\]
so its determinant is $\lambda_2(g)\cdot \lambda_1(g)\cdot \lambda(g)=1,$ for each $g\in G,$ and $S(V)^G$ is Gorenstein.\\\\
Assume now that $T(G)=<\tau>.$ If $c=0,$ then we are exactly in the first case.So we may assume that $c\neq 0.$ Consider $w_1':=cw_1+dv_0,$ then $\{w_2,w_1',v_0\}$ is a basis of $V,$with $\{w_1',v_0\}$ a basis of $W.$ Then the matrix of $\tau$ with respect to this new basis is 
\[\tau=\left(
\begin{array}{ccc}
1& 1& 0\\
0& 1&0 \\
0 & 0 & 1\end{array} \right).\]
Hence $S(V)^{<\tau>}=\mathbb{F}_p[z,w_1',v_0],$where $z=w_2^p-(w_1')^{p-1}w_2.$\\\\
By the same discussion as above,we have:
  \[(7.d)\left\{
\begin{array}{llll}
g(z)=\lambda_2(g)^pw_2^p+\delta_{12}(g)^p(w_1')^p+\delta_{13}(g)^pv_0^p+g(-(w_1')^{p-1})(\lambda_2(g)w_2+\delta_{12}(g)w_1'+\delta_{13}(g)v_0).\\

g(z)=\beta_1z+\beta_2(w_1')^p+\beta_3v_0^p+$mixed terms in $ \{w_1',v_0\}.\indent \beta_1,\beta_2,\beta_3 \in \mathbb{F}_p.\\
g(w_1')=\lambda_1(g)w_1'+\delta_{23}(g)v_0.\\
\end{array}
  \right.
  \]
Hence by comparing coefficient of $w_2^p$ in $g(z)$ , we get $\beta_1=\lambda_2(g)^p.$   
Let $m=(z,w_1',v_0)$ be the graded maximal homogenous ideal of $S(V)^{T(G)}.$ So $g$ acts linearly on $m/m^2$ and is represented by the matrix:

 \[ \left(
\begin{array}{ccc}
\lambda_2(g)^p& 0& 0\\
0& \lambda_1(g)& \delta_{23}(g)\\
0 & 0 & \lambda(g)\end{array} \right)= \left(
\begin{array}{ccc}
\lambda_2(g)& 0& 0\\
0& \lambda_1(g)& \delta_{23}(g)\\
0 & 0 & \lambda(g)\end{array} \right).\]
Hence by [Theorem \ref{T2}] $S(V)^G$ is Gorenstein.

\end{proof}

\begin{corollary}\label{G2}
If case $\textbf{(ii)}$ holds.Then $S(V)^G$ is Gorenstein.
\end{corollary}
\begin{proof}

Since $|T(G)|=p^2,$ it is elementary abelian,hence
 \[T(G)=\left(
\begin{array}{ccc}
1& \mathbb{F}_p& \mathbb{F}_p\\
0& 1& 0\\
0 & 0 & 1\end{array} \right)\indent or \indent T(G)=  \left(
\begin{array}{ccc}
1& 0& \mathbb{F}_p\\
0& 1& \mathbb{F}_p\\
0 & 0 & 1\end{array} \right).\]
Assume firstly that,
\[T(G)=\left(
\begin{array}{ccc}
1& \mathbb{F}_p& \mathbb{F}_p\\
0& 1& 0\\
0 & 0 & 1\end{array} \right). \] 
Hence by \textbf{Proposition \ref{FP1}}, $S(V)^{T(G)}=S(V)^{Fix_{T(G)}(W)}=S(W)[t]$ is a polynomial ring  where $t=w_2^{p^2}+q_1w_2^p+q_2w_2,$ and $q_1,q_2\in \mathbb{F}_p[w_1,v_0]$  is  homogenous of degrees $p^2-p,p^2-1$ respectively. \\
Let $m=(t,w_1,v_0)$ be the graded maximal homogenous ideal of $S(V)^{T(G)}.$By a similar discussion to the one  in $\textbf{(7.d)}$ we get that  $g$ acts linearly on $m/m^2$ and is represented by the matrix:
 \[ \left(
\begin{array}{ccc}
\lambda_2(g)^{p^2}& 0& 0\\
0& \lambda_1(g)& \delta_{23}(g)\\
0 & 0 & \lambda(g)\end{array} \right)= \left(
\begin{array}{ccc}
\lambda_2(g)& 0& 0\\
0& \lambda_1(g)& \delta_{23}(g)\\
0 & 0 & \lambda(g)\end{array} \right).\]
Consequently,by [Theorem \ref{T2}] $S(V)^G$ is Gorenstein.\\\\
Assume now that
\[T(G)=\left(
\begin{array}{ccc}
1& 0& \mathbb{F}_p\\
0& 1& \mathbb{F}_p\\
0 & 0 & 1\end{array} \right).\]
Hence by the \textbf{Corollary \ref{ECa}}, $S(V)^{T(G)}=\mathbb{F}_p[y,x,v_0]$ is a polynomial ring, where $y=w_2^p-v_0^{p-1}w_2$ and $x=w_1^p-v_0^{p-1}w_1.$ We have,using $F=\mathbb{F}_p$ (so $\lambda(g)^{p-1}=1$):\\
\[(7.e)\left\{
\begin{array}{ll}
 g(y)=(\lambda_2(g)w_2^p+\delta_{12}(g)w_1^p+\delta_{13}(g)v_0^p)-v_0^{p-1}(\lambda_2(g)w_2+\delta_{12}(g)w_1+\delta_{13}(g)v_0)=\lambda_2(g)y+\delta_{12}(g)x .\\
 g(x)=(\lambda_1(g)w_1^p+\delta_{23}(g)v_0^p)-v_0^{p-1}(\lambda_1(g)w_1+\delta_{23}(g)v_0)=\lambda_1(g)x. \\

 \end{array}
  \right.
  \]
Let $m=(y,x,v_0)$ be the graded maximal homogenous ideal of $S(V)^{T(G)}.$ So $g$ acts linearly on $m$ (resp. $m/m^2$) and is  represented by the matrix:

 \[ \left(
\begin{array}{ccc}
\lambda_2(g)& \delta_{12}(g)& 0\\
0& \lambda_1(g)& 0\\
0 & 0 & \lambda(g)\end{array} \right).\]
Consequently,by [Theorem \ref{T2}] $S(V)^G$ is Gorenstein.

\end{proof}

\begin{corollary}\label{G3}

If case $\textbf{(iii)}$ holds.Then $S(V)^G$ is Gorenstein.

\end{corollary}
\begin{proof}
Since $|T(G)|=p^3,$

 \[T(G)=\left(
\begin{array}{ccc}
1& \mathbb{F}_p& \mathbb{F}_p\\
0& 1& \mathbb{F}_p\\
0 & 0 & 1\end{array} \right).\]
Hence $T(G)=U_3(\mathbb{F}_p),$ the group of all upper triangular matrices with $1'$s along the diagonal 
and entries from the finite field $\mathbb{F}_p.$ This group was shown to have a polynomial ring of invariants by Bertin \cite{Bertin}.\\\\
We shall find explicitly the generators of $S(V)^{T(G)}.$
$T(G)=<\sigma_1,\sigma_2,\sigma_3>,$ where
 \[\sigma_1=\left(
\begin{array}{ccc}
1& 0& 1\\
0& 1& 0\\
0 & 0 & 1\end{array} \right),\sigma_2=\left(
\begin{array}{ccc}
1&1& 0\\
0& 1& 0\\
0 & 0 & 1\end{array} \right),\sigma_3=\left(
\begin{array}{ccc}
1& 0& 0\\
0& 1& 1\\
0 & 0 & 1\end{array} \right).\]
Consequently, by \cite[Lemma 3.13]{Braun3},
$S(V)^{<\sigma_1>}=\mathbb{F}_p[a,w_1,v_0],$is a polynomial ring where $a=w_2^p-v_0^{p-1}w_2.$Define $u=a^p-b^{p-1}a,$ where $b:=w_1^p-v_0^{p-1}w_1.$ Observe that  $\sigma_3(b)=b,\sigma_3(a)=a,\sigma_2(b)=b,\sigma_2(a)=a+b,$hence $\sigma_2(u)=u.$ We have, for $i=1,2,3:$\\

\[\left\{
\begin{array}{ll}
 \sigma_i(u)=u .\\
 \sigma_i(b)=b. \\
  \sigma_i(v_0)=v_0. \\

 \end{array}
  \right.
  \]

The $T(G)-$invariants  $\{u,b,v_0\}$ are  algebraically independent  over
$\mathbb{F}_p.$Since $deg(u)\cdot deg(b)\cdot deg(v_0)=p^2\cdot p \cdot 1=p^3=|T(G)|,$we have by \cite[Theorem 3.9.4]{Kemper1} that $S(V)^{T(G)}=\mathbb{F}_p[u,b,v_0]$ is a
polynomial ring.Let $m=(u,b,v_0)$ be the graded maximal homogenous ideal of $S(V)^{T(G)}.$\\
Since $g(b)=\lambda_1(g)^pw_1^p+s=\alpha b+\beta v_0^p,$ where $s\in m^2,$and $\alpha,\beta\in \mathbb{F}_p,$ we get that $\alpha=\lambda_1(g)^p.$Since $g(u)=\lambda_2(g)^{p^2}u+r=\gamma u + \delta b^p+\pi v_0^{p^2}+$ mixed terms in $\{b,v_0\},$ where $r\in m^2,$and $\gamma,\delta,\pi\in \mathbb{F}_p,$ we get that $\gamma=\lambda_2(g)^{p^2}.$Hence,

 \[g|m/m^2= \left(
\begin{array}{ccc}
\lambda_2(g)^{p^2}& 0& 0\\
0& \lambda_1(g)^{p}& 0\\
0 & 0 & \lambda(g)\end{array} \right)=\left(
\begin{array}{ccc}
\lambda_2(g)& 0& 0\\
0& \lambda_1(g)& 0\\
0 & 0 & \lambda(g)\end{array} \right).\]
Consequently,by [Theorem \ref{T2}] $S(V)^G$ is Gorenstein.

\end{proof}

\newpage

\section{Chapter($G$): $V$ is indecomposable and contains (at least) two $1-$dimensional $G-$ submodules.}
Let $Fw_1,Fw_2$ be the two $1-$dimensional $G-$submodules. Hence $W=Fw_1+Fw_2$ is 
$G-$submodule of $V.$
Let $v\in V\backslash W,$ so $\{v,w_1,w_2\}$ is a basis of $V.$ Each $g\in G$ is represented with respect to this basis by the matrix :
 \[g=\left(
\begin{array}{ccc}
\lambda(g)& \delta_{12}(g)& \delta_{13}(g)\\
0& \lambda_2(g)& 0\\
0 & 0 & \lambda_1(g)\end{array} \right).\]
In particular each transvection $\sigma\in T(G)$ is represented by the matrix :
 \[\sigma=\left(
\begin{array}{ccc}
\lambda(\sigma)& \delta_{12}(\sigma)& \delta_{13}(\sigma)\\
0& 1& 0\\
0 & 0 & 1\end{array} \right).\] Hence $T(G)|W=Id.$ Hence 
 \[T(G)\subseteq \left(
\begin{array}{ccc}
1& F& F\\
0& 1& 0\\
0 & 0 & 1\end{array} \right).\] 
Therefore by \textbf{Proposition \ref{FP1}} $S(V)^{T(G)}=S(W)[z],$ where $z=v^{p^t}+q_1v^{p^{t-1}}+...+q_tv ,$ an homogenous polynomial, where $q_i\in S(W)^{T(G)|W}=S(W)$
and $|T(G)|=p^t.$ Let $m=(z,w_2,w_1).$ Then $m/m^2=F\bar{z}+F\bar{w_2}+F\bar{w_1}.$
Now for $g\in G$ we have $g(z)=\alpha z+x$ where $x$ is in $m^2.$ Therefore $\bar{g}(\bar{z})=\alpha \bar{z},$ where $\bar{g}=gT(G).$ Also $g(z)=g(v)^{p^t}+g(q_1)g(v)^{p^{t-1}}+...+g(q_t)g(v)=\lambda(g)^{p^t}v^{p^t}+$ mixed terms. Hence $\alpha=\lambda(g)^{p^t}$ and therefore $\bar{g}$ is represented on $F\bar{z}+F\bar{w_2}+F\bar{w_1}$ by the matrix :

 \[\left(
\begin{array}{ccc}
\lambda(g)^{p^t}& 0& 0\\
0& \lambda_2(g)& 0\\
0 & 0 & \lambda_1(g)\end{array} \right).\] Therefore by [Theorem \ref{T2}] $S(V)^G$ is Gorenstein iff $\lambda(g)^{p^t}\cdot \lambda_2(g)\cdot \lambda_1(g)=1,$ for each $g\in G.$

\begin{corollary}\label{FF1}
Suppose $V=\mathbb{F}_p^3 ,$ contains two $1-$ dimensional $G-$submodules and $G\subseteq SL(3,\mathbb{F}_p).$ Then $S(V)^G$ is Gorenstein.

\end{corollary}

\begin{proof}

$\lambda(g)\in \mathbb{F}_p$ hence $\lambda(g)^{p^t}\cdot \lambda_2(g)\cdot \lambda_1(g)=\lambda(g)\cdot \lambda_2(g)\cdot \lambda_1(g)=det(g)=1,$ where the last equaliry holds since $g\in SL(3,\mathbb{F}_p) . $

\end{proof}

\newpage

\section{References}

\bibliographystyle{plain}
\bibliography {References}

\end{document}